\documentclass[12pt]{article}
\usepackage[margin=1.2in]{geometry}
\usepackage{amssymb,amsmath,amscd,amsthm,pb-diagram,tikz-cd,enumerate,hyperref}

\usepackage{mathtools}
\usepackage{ stmaryrd }

\usepackage{pinlabel}

\title{Chain duality for categories over complexes}

\date{}

\newcommand{\cC}{{\mathcal{C}}}

\newcommand{\cF}{{\mathcal{F}}}

\newcommand{\bfx}{{\mathbf x}}

\newcommand{\A}{{\mathbb A}}

\newcommand{\R}{{\mathbb R}}
\newcommand{\Z}{{\mathbb Z}}
\newcommand{\ol}[1]{{\overline{#1}}}
\newcommand{\ul}[1]{{\underline{#1}}}
\newcommand{\wh}[1]{{\widehat{#1}}}

\DeclareMathOperator{\ch}{{Ch}}

\DeclareMathOperator{\cone}{{Cone}}

\DeclareMathOperator{\ev}{{ev}}
\DeclareMathOperator{\Hom}{{Hom}}
\DeclareMathOperator{\id}{Id}

\DeclareMathOperator{\Mod}{-mod}
\DeclareMathOperator{\mor}{Mor}

\DeclareMathOperator{\op}{{op}}

\DeclareMathOperator{\sw}{{sw}}
\DeclareMathOperator{\tot}{{Tot}}

\DeclareMathOperator{\YL}{{YL}}
\DeclareMathOperator*{\colim}{{colim}}
\newtheorem{theorem}{Theorem}
\newtheorem{proposition}[theorem]{Proposition}

\newtheorem{corollary}[theorem]{Corollary}
\newtheorem{lemma}[theorem]{Lemma}
\theoremstyle{definition}
\newtheorem{definition}[theorem]{Definition}
\newtheorem{remark}[theorem]{Remark}
\newtheorem{example}[theorem]{Example}

\author{James F. Davis,  Carmen Rovi}

\newcommand{\Addresses}{{
\bigskip
\footnotesize

James F.~Davis,  \textsc{Department of Mathematics, Indiana University, Rawles Hall, 831 East 3rd St,
Bloomington, IN 47405} \par \nopagebreak
\textit{E-mail address}, J.F.~Davis:  \texttt{jfdavis@indiana.edu}

\medskip

Carmen Rovi,  \textsc{Department of Mathematics \& Statistics, Loyola University Chicago, 1032 W. Sheridan Rd.,
Chicago, IL 60660} \par \nopagebreak
\textit{E-mail address}, C.~Rovi:  \texttt{crovi@luc.edu}

}}

\begin{document}
\maketitle

\begin{center}
\textit{Dedicated to Dennis Sullivan on the occasion of his 80th birthday.}
\end{center}

\begin{abstract}

We show that the additive category of chain complexes parametrized by a finite simplicial complex $K$ forms a category with chain duality.   This fact, never fully proven in the original reference (Ranicki, 1992), is fundamental for Ranicki's algebraic formulation of the surgery exact sequence of Sullivan and Wall, and his interpretation of the surgery obstruction map as the passage from local Poincar\'e duality to global Poincar\'e duality.  

Our paper also gives a new, conceptual, and geometric treatment of chain duality on $K$-based chain complexes.

\end{abstract}


\section{Introduction}

Kervaire and Milnor \cite{KM} developed and applied the new field of surgery to classify exotic smooth structures on spheres.   Browder and Novikov independently extended and relativized the theory.  Sullivan in his thesis \cite{Sullivan-thesis}  investigated the obstruction theory for  deforming a homotopy equivalence to a homeomorphism.   In seminar notes \cite{Sullivan-Notes}    written shortly after his thesis,  Sullivan's Theorem 3 packaged  this in what is now called the surgery exact sequence.  (We will be ahistorical and concentrate on topological manifolds; Kervaire-Milnor concentrated on smooth manifolds and Sullivan on PL-manifolds.   The extension to topological manifolds is due to the deep work of Kirby and Siebenmann \cite{KS}.). It was extended to the nonsimply-connected case and to the case of compact manifolds by Wall \cite{Wall}.   The surgery exact sequence for a closed $n$-dimensional manifold $X$ with $n \geq 5$ is
$$
\dots \to L_{n+1} (\Z[\pi_1(X)]) \to   \mathbb{S}^{TOP}(X) \to{\cal N}^{TOP}(X)  \xrightarrow{\sigma}  L_{n} (\Z[\pi_1X]). 
$$
The object one wants to compute is the structure set $\mathbb{S}^{TOP}(X)$, first defined by Sullivan.   Representatives of the structure set are given by (simple) homotopy equivalences from a closed $n$-manifold to $X$.  Computing the structure set is the key ingredient in computing the manifold moduli set, the set of homeomorphism types of  $n$-manifolds homotopy equivalent to $X$.   The beauty of the surgery exact sequence is marred by many flaws.   One is that it is an exact sequence of pointed sets.   Another flaw is standard  with a long exact sequence, to do computations one needs to compute the surgery obstruction map $\sigma$, including its domain, the normal invariants ${\cal N}^{TOP}(X)$, and its codomain, the $L$-groups.    For many fundamental groups (e.g.~finite groups or finitely generated abelian groups) one can compute the $L$-groups algebraically.   Sullivan, in his thesis, analyzed the normal invariants, using transversality to establish a bijection  ${\cal N}^{TOP}(X)   \cong [X,G/TOP]$.  He also computed the homotopy groups $\pi_i(G/TOP)$ which vanish for $i$ odd, have order 2 for $i \equiv 2 \pmod 4$, and are infinite cyclic for $i \equiv 0 \pmod 4$ (this follows from the generalized Poincar\'e conjecture).   Furthermore, Sullivan analyzed the homotopy type of $G/TOP$ and established Sullivan periodicity, $\Omega^4 (\Z \times G/TOP) \simeq \Z \times G/TOP$.  This then determines an $\Omega$-spectrum $\mathbb{L}_.$ and its 1-connective cover $\mathbb{L}_.\langle 1 \rangle$.   There is a formal identification  $[X,G/TOP] = H^0(X; \mathbb{L}_.\langle 1 \rangle)$.   

The surgery exact sequence is now becoming more presentable, but it is still marred by functoriality issues: the $L$-groups are covariant in $X$, the normal invariants are contravariant in $X$, and the structure set has no obvious variance at all.    Furthermore it is only defined for manifolds, and one would like an abelian group structure. These flaws make computing the surgery obstruction map difficult.     Quinn's \cite{Quinn} vision (largely carried out by Ranicki \cite{bluebook}, see also \cite{KMM}) is to find a bijection between the surgery exact sequence and a long exact sequences of abelian groups defined for every space $X$ and fully covariant in $X$.
In more detail, there is the following commutative diagram

$$
\begin{tikzcd}
\arrow[r]  & L_{n+1} (\Z[\pi_1X]) \arrow[r] \arrow[d, "="] &  \mathbb{S}^{TOP}(X) \arrow[r] \arrow[d, "\cong"] & {\cal N}^{TOP}(X)  \arrow[r] \arrow[d, "\cong"] & L_{n} (\Z[\pi_1X]) \arrow[d, "="] \\
\arrow[r] & L_{n+1} (\Z[\pi_1X]) \arrow[r] \arrow[d, "="]&   \mathbb{S}^{\langle 1 \rangle}_{n+1}(X) \arrow[r] \arrow[d] & H_n(X; \mathbb{L}_.\langle 1 \rangle) \arrow[r,"A\langle 1 \rangle" ] \arrow[d] & L_{n} (\Z[\pi_1X]) \arrow[d, "="] \\
\arrow[r] & L_{n+1} (\Z[\pi_1X]) \arrow[r] &   \mathbb{S}_{n+1}(X) \arrow[r] & H_n(X; \mathbb{L}_.) \arrow[r, "A"]  & L_{n} (\Z[\pi_1X]) \\
\end{tikzcd}
$$
where the vertical maps labelled $\cong$ are bijections when $X$ is a closed $n$-manifold and the bottom two horizontal lines are exact sequences of abelian groups, defined for any space $X$. These two lines are called the 1-connective algebraic surgery exact sequence and the algebraic surgery exact sequence, respectively.  The maps $A\langle 1 \rangle$ and $A$ are called assembly maps; they are defined at the spectrum level. Hence there is a long exact sequence
of homotopy groups, where the algebraic structure groups are defined be to the homotopy groups of the cofiber of the assembly maps.    The map $A$ is conjectured to be an isomorphism when $X = B\pi$ with $\pi$ torsionfree. 

There are  myriad ways of constructing the assembly maps (the construction in \cite{DL98} seems best for computations).  Different constructions are identified via axiomatics (see \cite{WW}, also \cite{DL98}).   Ranicki's version of assembly, needed for his approach to the above diagram, was motivated by his earlier work with Weiss \cite{RW} viewing the assembly maps as a passage from local to global Poincar\'e duality.     Much earlier Ranicki \cite{ats1} reinterpreted Wall's algebraic $L$-groups as bordism groups of algebraic Poincar\'e complexes over the group ring $\Z[\pi_1X]$.  This is the global Poincar\'e duality.  The local Poincar\'e duality comes from making a geometric degree one normal map transverse to the dual cones of $X$ (see Section \ref{sec:K-based} for the definition).   These degree one normal maps to the cones  are then assembled to give the original degree one normal map.

More precisely, Ranicki \cite{bluebook} defined the notion of an additive category with chain duality $\A$, the associated algebraic bordism category $\Lambda(\A)$ (see his Example 3.3), and the corresponding $L$-groups $L_n(\A)$ (see his Definition 1.8).  In his notation, the assembly map is given by establishing a map of algebraic bordism categories (see his Proposition 9.11)
$$
\Lambda((\Z,X)\Mod)) \to \Lambda(\Z[\pi_1X]\Mod)
$$
and defining the assembly map to be the induced map on $L$-groups.  However, one flaw in his argument is that he never provided a proof that $(\Z,X)\Mod$ is an additive category with chain duality, despite his assertion in Proposition 5.1 of \cite{bluebook}.   Our modest contribution to this saga is to provide a self-contained, conceptual, and geometric proof that $(\Z,X)\Mod$ is an additive category with chain duality.   

We are not the first to provide a proof of this result -- one is given in Section 5 of \cite{Spiros-Tibor}.   However, we found the proof and its notation rather dense.    Another account of this result is given in a recent preprint of Frank Connolly \cite{Frank}.   Although his aims are quite similar to ours, the approach is  different, the reader may wish to compare.   

We now outline our paper. In Section \ref{sec:CD} we review Ranicki's notion of an additive category with chain duality, this is an additive category with a chain duality functor satisfying a chain homotopy equivalence condition.   In Section \ref{sec:K-based} we fix a finite simplicial complex $K$ (e.g. a triangulation of a compact manifold), and we define Ranicki's additive categories of $K$-based chain complexes.   Here we need to warn the reader that we have deviated from Ranicki's notation in \cite{bluebook}, which we found difficult to use.   A comparison between our notation and Ranicki's is given in Remark \ref{R_notation}.  The two key additive categories are $\ch(\Z(K)\Mod)$ and $\ch(\Z(K^{\op})\Mod)$.   The latter category is the one whose $L$-theory gives the normal invariants, so is perhaps more important.   The simplicial chain complex $\Delta K$ gives an object of  $\ch(\Z(K)\Mod)$ and the simplicial cochain complex $\Delta K^{-*}$ gives an object of $\ch(\Z(K^{\op})\Mod)$.  More generally, given a CW-complex $X$ with a $K$-dissection, the cellular chains $C(X)$ give an object of $\ch(\Z(K)\Mod)$ and given a CW-complex $X$ with a $K^{\op}$-dissection, the cellular chains $C(X)$ give an object of $\ch(\Z(K^{\op})\Mod)$.    We related this to dual cell decompositions, defined even when $K$ is not a manifold.  In Section \ref{sec:cat_point_of_view}, we develop homological algebra necessary for our proof that these categories admit a chain duality.

Section \ref{sec:dual_cell} may be of independent interest.    For a finite simplicial complex $K$, we define the dual cell decomposition $DK$ which is a regular CW-complex refining the simplicial structure on $K$.  Corollary \ref{varepsilon is a weak eq} and Remark \ref{two-sided} say that this, in some sense, gives a two-sided bar resolution for the category of posets of $K$. 

Finally, in Section \ref{K-based chain duality} we define chain duality functors $(T,  \tau)$ on $\ch(\Z(K^{\op})\Mod)$ and $\ch(\Z(K)\Mod)$ and prove our main theorem.

\begin{theorem} \label{main}
 The following are  additive categories with  chain duality
$$\begin{cases}  (\ch(\Z(K^{\op})\Mod),T, \tau), \\ (\ch(\Z(K)\Mod),T, \tau).   \end{cases}$$
\end{theorem}

\section{Chain Duality} \label{sec:CD} 

For a category $K$, write $\sigma \in K$ when $\sigma$ is an object of $K$ and $K(\sigma,\tau)$ for the set of morphisms from $\sigma$ to $\tau$.   A {\em preadditive category} is a category where all morphism sets are abelian groups and composition is bilinear.    An {\em additive category} is a preadditive category which admits finite products and coproducts.   An example of an additive category is the category of finitely generated free abelian groups.   

Let $\A$ be an additive category and let $\ch(\A)$ be the category of finite chain complexes over $\A$ where {\em finite} means that $C_n = 0$ for all but a finite number of $n$.   Homotopy notions make sense in this category:  the notions of two 
chain
 maps being chain homotopic, a 
chain map being a chain homotopy equivalence, two chain complexes being chain homotopy equivalent, and a chain complex being contractible. 
The notion of homology of a chain complex over an additive category does not make sense.

Let $\ch_{\bullet,\bullet}(\A)$ be the category of finite bigraded chain complexes over $\A$.   There are functors
\begin{align*}
\tot & : \ch_{\bullet,\bullet}(\A) \to \ch(\A) \\
\Hom_{\bullet,\bullet} & : \ch(\A)^{\op} \times  \ch(\A) \to \ch_{\bullet,\bullet}(\Z\Mod)
\end{align*}
where $\tot(C_{\bullet,\bullet})_n = \bigoplus_{p+q = n}  C_{p+q}$ and 
$\Hom(C,D)_{p,q} = \A(C_{-p},D_q)$.  (Throughout this paper, if the differentials are standard or can be easily determined, we omit them for readability).
   If $C$ and $D$ are finite chain complexes over an additive category $\A$, then
$$
\Hom_{\A}(C,D) := \tot(\Hom_{\bullet,\bullet}(C,D))
$$
is a chain complex of abelian groups with differentials
\begin{align*}
d_{\Hom_{\A}(C,D)} :  \Hom_{\A}(C,D)_n & \to  \Hom_{\A}(C,D)_{n-1} \\
d (f) & =  d_D \circ f + (-1)^{n+1} f \circ d_C .
\end{align*}

A 0-cycle is a chain map; the difference of chain maps is a boundary if and only if the chain maps are chain homotopic.   In particular, there is a monomorphism of abelian groups $\ch(\A)(C, D) \to \Hom_{\A}(C,D)_0$.

If $C$ and $D$ are chain complexes of abelian groups, then there is a chain complex $(C \otimes D, d_{\otimes})$ with differentials 
\begin{align*}
d_{\otimes} : (C \otimes D)_n & \to (C \otimes D)_{n-1} \\
d_{\otimes} ( x \otimes y)  & = d_C (x) \otimes y + (-1)^{|x|}  x \otimes d_D(y).
\end{align*}
\begin{definition}

A {\em chain duality functor} $(T, \tau)$ on an additive category $\A$ is an additive functor $T: \ch(\A) \to \ch(\A)^{\op}$ together with a natural chain map
$$
\tau_{C,D}: \Hom_{\A}(TC,D) \to \Hom_{\A}(TD,C)
$$
defined for each pair of chain complexes $C,D \in \ch(\A)$ so that $\tau^2 = \id$ in the sense that $\tau_{D,C} \circ  \tau_{C,D} = \id$.
\end{definition}

\begin{remark}
By restricting to 0-cycles, the natural chain map $\tau$ induces a natural isomorphism of abelian groups
$$
\tau_{C,D} : \ch(\A)(TC,D) \to \ch(\A)(TD,C).
$$
\end{remark}

\begin{lemma}  \label{ec_appears}
Let $(T, \tau)$ be a chain duality functor on $\A$.  For $C \in \ch(\A)$, let $e_C : T^2C \to C$ be $\tau(\id_{TC})$.   This defines a natural transformation 
$$
e: T^2 \to \id : \ch(\A) \to \ch(\A) 
$$
so that for each object $C \in \ch(\A)$,
$$
e_{TC} \circ T(e_C) = \id_{TC}: TC \to T^3C \to TC.
$$
\end{lemma}

\begin{proof}
Suppose $\alpha : TU \to V$ and $\beta : V \to W$ are chain maps.    Then naturality of $\tau$ implies that 
$$
\tau(\beta \circ \alpha) = \tau(\alpha) \circ T(\beta).
$$
Thus
\begin{align*}
e_{TC} \circ T(e_C) & = \tau(\id_{T^2C}) \circ T(e_C)\\
& = \tau(e_C \circ \id_{T^2C}) \\
& = \tau(e_C) \\
& = \id_{TC}.
\end{align*}
\end{proof}

It is also true, conversely, that an additive functor $T: \ch(\A) \to \ch(\A) ^{\op}$ and natural transformation $e: T^2 \to \id$ satisfying 
$
e_{TC} \circ T(e_C) = \id_{TC}
$
for all $C \in \ch(\A)$ determines a chain duality functor $(T,\tau)$ where $\tau(f : TC \to D) := e_C \circ T(f)$, but we omit the proof of this fact.

\begin{definition}
 A {\em  chain duality} on an additive category $\A$ is a chain duality functor $(T, \tau)$ so that $e_C : T^2C \to C$ is a chain homotopy equivalence for all $C \in \ch(\A)$.
\end{definition}

This is equivalent to
 the definition in  Andrew Ranicki's book \cite[ Definition 1.1]{bluebook}.   
Notice that a chain duality functor does not necessarily give a chain duality, because of the extra condition that $e_C$ is a chain homotopy equivalence.
 We separately defined  a chain duality functor because there can be uses for the weaker notion, for example, see the thesis of Christopher Palmer \cite{P15}.
 
\section{$K$-based chain complexes} \label{sec:K-based} 

Let $K$ be a finite set.   

\begin{definition}
 An abelian group $M$ is {\em $K$-based} if it is expressed as a direct sum
 $$
 M = \bigoplus_{\sigma \in K} M(\sigma).
 $$
 A morphism $f : M \to N$ of $K$-based abelian groups is simply a homomorphism of the underlying abelian groups $M$ and $N$.   Equivalently, it is a collection of homomorphisms $\{ M(\sigma) \to N(\tau)) \mid  \sigma, \tau \in K \}$.
\end{definition}

In our exposition, we choose to work with $M$ being an abelian group.   However, everything we say (and everything Ranicki says in \cite{bluebook}) generalizes to the context of $R$-modules where $R$ is a ring with involution.

When the set $K$ is a finite poset, we are interested in a subcategory of the $K$-based abelian groups.

\begin{definition}  \label{finite poset}
Let $K$ be a finite poset.

 The objects of $\Z(K)\Mod$ are the $K$-based 
abelian groups
 $$M = \oplus_{\sigma \in K} M(\sigma)$$
  where $M(\sigma)$ is a finitely generated free abelian group for each $\sigma \in K$.   A $K$-based morphism $f: M \to N$ is a morphism in $\Z(K)\Mod$ if, for all $\tau \in K$,
$$
f(M(\tau)) \subset \bigoplus_{\sigma \leq \tau} N(\sigma).
$$
\end{definition}

The slogan for morphisms is ``bigger to smaller."   

Let $K$ be a finite simplicial complex.    There is an associated poset, also called $K$, whose objects are the simplices of $K$ and whose morphisms are inclusions:   $\sigma \leq \tau$ means $\sigma \subseteq \tau$.   Our quintessential examples of a poset will be either $K$ or $K^{\op}$.  Our convention will be that $\sigma \leq \tau$ means $\sigma \leq_K \tau$ and we will try and minimize the use of $\tau \leq_{K^{\op}} \sigma$.
 The simplicial chain complex $\Delta(K) \in \ch(\Z(K)\Mod)$ illustrates the bigger-to-smaller slogan.   Here $\Delta(K)_n = \oplus_{\sigma \in K^n} \Delta(K)_n(\sigma)$, with
 $$ \Delta(K)_n(\sigma) \cong \begin{cases} \Z & \textnormal{if } n = |\sigma| \\0 & \textnormal{otherwise} \end{cases} $$

Since duality is the fundamental feature of this paper, we introduce it immediately.

\begin{definition}  Let $K$ be a finite poset.
 The duality functor 
 $$
 * : \Z(K)\Mod \to (\Z(K^{\op})\Mod)^{\op}
 $$
 is defined on objects by
 $$
 M^* = \bigoplus_{\sigma \in K} M(\sigma)^*
 $$
 where $M(\sigma)^* = \Hom_\Z(M(\sigma),\Z)$.  There is a natural isomorphism \\ $E : \id \Rightarrow **$ with $E_M : M \to M^{**}$ induced by $E_M(\sigma) : M(\sigma) \to M(\sigma)^{**}$ given by $m \mapsto (\phi \mapsto \phi(m))$.  The duality functor and natural isomorphism extend to chain complexes
 $$
 -* : \ch(\Z(K)\Mod) \to \ch(\Z(K^{\op})\Mod)^{\op}
 $$
 with $(C^{-*})_n = C^{-n} := (C_n)^*$.
\end{definition}

This definition illustrates some of our notational conventions.   
    We write $C$ (and not $C_*$) to denote a chain complex.  We use $C^{-*}$ (and not $C^*$) so that the dual is also a chain complex, whose differential has degree minus one.  
 There are also sign conventions on the differential; we follow the sign conventions of Dold \cite{Dold}:  the differential $(C_{-*})_{n+1} \to (C_{-*})_{n}$ is given by $(-1)^{n+1} (\partial_{-n})^*$.

  The simplicial cochain complex $\Delta K^{-*} \in \ch(\Z(K^{\op})\Mod)$ of a finite  simplicial complex illustrates the $K^{\op}$-slogan  ``smaller-to-bigger."

\begin{definition} \label{def_dissection}
Let $K$ be a finite simplicial complex and let $X$ be a finite CW-complex.
\begin{enumerate}
 
\item A $K$-dissection of $X$ is a collection $\{X(\sigma) \mid \sigma \in K\}$ of subcomplexes of $X$ so that
\begin{enumerate}
 
 \item $X(\sigma) \cap X(\rho) = 
\begin{cases}
X(\sigma \cap \rho) & \text{if } \sigma \cap \rho \in K \\
\emptyset & \text{ otherwise}
\end{cases}$
 
 \item $X =\displaystyle \bigcup_{\sigma \in K} X(\sigma)$.   
 
\end{enumerate}

\item A $K^{\op}$-dissection of $X$ is a collection $\{X(\sigma) \mid \sigma \in K\}$ of subcomplexes of $X$ so that
\begin{enumerate}
 
 \item $X(\sigma) \cap X(\rho) = 
\begin{cases}
X(\sigma \cup \rho) & \text{if } \sigma \cup \rho \in K \\
\emptyset & \text{ otherwise}
\end{cases}$
 
 \item $X =\displaystyle \bigcup_{\sigma \in K} X(\sigma)$.   
 
\end{enumerate}
\end{enumerate}
\end{definition}
Here $\sigma \cup \rho$ is the smallest simplex of $K$ which contains $\sigma$ and $\rho$, if it exists.  Note that in a $K$-dissection $\sigma \leq \tau$ implies that $X(\sigma) \subset X(\tau)$, while in a $K^{\op}$-dissection, $\sigma \leq \tau$ implies that $X(\tau) \subset X(\sigma)$.

\begin{remark}
The $K^{\op}$-dissections described here are $K$-dissections in Ranicki's terminology.
\end{remark}

\begin{example} \label{K_dissected_by_K}
The geometric realization of a finite simplicial complex $K$ has both a $K$-dissection given by the geometric realization of the simplices and a $K^{\op}$-dissection given by the dual cones of simplices.  We describe the latter in order to fix notation.

   Let $K'$ be the barycentric subdivision of $K$.   The vertices of $K'$ are the barycenters $\ol {\sigma}_i$ of the geometric realization of the simplices $\sigma_i \in K$.   An $r$-simplex in $K'$ is given by a sequence $\ol{\sigma}_0 \ol{\sigma}_1 \dots \ol{\sigma}_r$ where $\sigma_i < \sigma_{i+1}$, and   $K'$ is a subdivision of $K$ (see 
Chapter 3, Section 3 of Spanier  \cite{S66}  for the definition of a subdivision); in particular there is a PL-homeomorphism $|K'| \to |K|$.  For $\sigma \leq \tau \in K$, the {\em dual cell $D_{\tau} \sigma$} is the union of the geometric realization of all simplices $\ol{\sigma}_0 \ol{\sigma}_1 \dots \ol{\sigma}_r$ of the barycentric subdivision so that $\sigma \leq \sigma_0 < \sigma_1 < \dots < \sigma_r \leq \tau$.    Define the {\em dual cone of $\sigma \in K$} to be $D_K \sigma = \cup_{\{\tau \mid \sigma \leq \tau\}} D_\tau \sigma$.   Then $\{ D_K \sigma\}$ gives a $K^{\op}$-dissection of the geometric realization of $K$.
   
  With $K$ a $2$-simplex, Figure \ref{fig:dissection-2-simplex} shows a $K$ and $K^{\op}$ dissection of the geometric realization of a $2$-simplex.
      
   \begin{figure}[ht!]
\labellist
\small\hair 2pt
\pinlabel \small {\textit{ $K'$}} at 47 50
\pinlabel \small {\tiny{$\overline{ \sigma_{0}}$}} at 5  67
\pinlabel \small {\tiny{$\overline{ \sigma_{1}}$}} at 87  67
\pinlabel \small {\tiny{$\overline{ \sigma_{2}}$}} at 45  155
\pinlabel \small {\tiny{$\overline{ \sigma_{02}}$}} at 15  112
\pinlabel \small {\tiny{$\overline{ \sigma_{12}}$}} at 77  112
\pinlabel \small {\tiny{$\overline{ \tau}$}} at 52  108
\pinlabel \small {\tiny{$\overline{ \sigma_{01}}$}} at 47  67
\pinlabel \small {\textit{dual cells}} at 170 50
\pinlabel \small {\scriptsize{$D_{\tau} \sigma_{0}$}} at 150  85
\pinlabel \small {\scriptsize{$D_{\tau} \sigma_{1}$}} at 190  85
\pinlabel \small {\scriptsize{$D_{\tau} \sigma_{2}$}} at 170  115
\pinlabel \small {\scriptsize{$D_{\sigma_0}\sigma_0$}} at 130  67
\pinlabel \small { \textit{$K^{\op}$-dissection}} at 305 110
\pinlabel \small { \scriptsize{$ D_K\sigma_0$}} at 285 141
\pinlabel \small { \scriptsize{$ D_K\sigma_1$}} at 322 141
\pinlabel \small { \scriptsize{$ D_K\sigma_2$}} at 303 170
\pinlabel \small {\textit{$K$-dissection} } at 305 -10
\pinlabel \small {\scriptsize{\textit{$\sigma_0$} }} at 265 03
\pinlabel \small {\scriptsize{\textit{$\sigma_1$} }} at 352 03
\pinlabel \small {\scriptsize{\textit{$\sigma_2$} }} at 309 87
\pinlabel \small {\scriptsize{\textit{$\sigma_{02}$} }} at 273 43
\pinlabel \small {\scriptsize{\textit{$\sigma_{12}$} }} at 341 43
\pinlabel \small {\scriptsize{\textit{$\sigma_{01}$} }} at 309 03
\pinlabel \small {\small{\textit{$\tau$} }} at 308 35
\endlabellist
\centering
\includegraphics[width=130mm, height= 75mm]{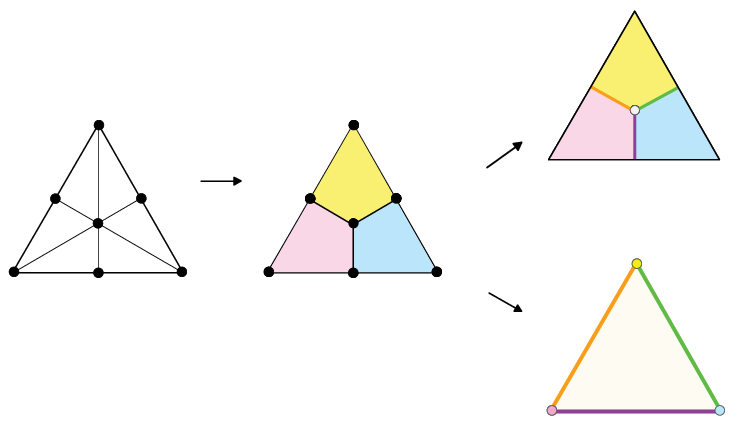}
\vspace{10pt}
\caption{Dual cells and $K$- and $K^{\op}$-dissections of a 2-simplex}
\label{fig:dissection-2-simplex}
\end{figure}
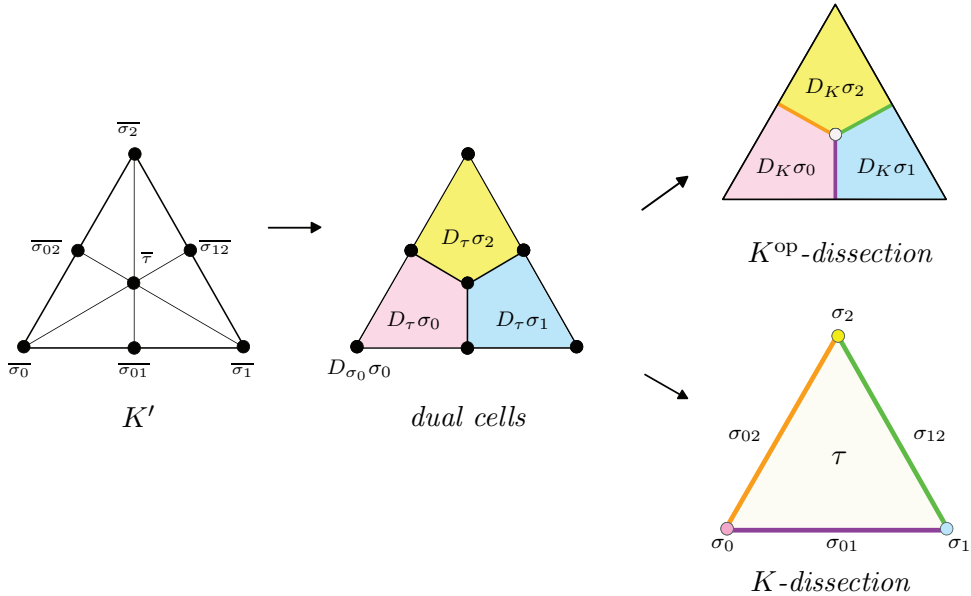
   
\end{example}

If $(Y,B)$ is a CW-pair, then the surjection of the cellular $n$-chains $C_nY \to C_n(Y,B)$ has a canonical splitting (informally, it is given by the span of the $n$-cells of $Y-B$).    We will thus consider $C_n(Y,B)$ to be a subgroup of $C_nY$.   More generally, given a CW-triple $(Z,Y,B)$, we consider $C_n(Y,B)$ to be a subgroup of $C_nZ$.

If $X$ has a $K$-dissection and $\sigma \in K$, define 
 $\partial X(\sigma) = \cup_{\rho < \sigma} X(\rho)$.   This is a subcomplex of $X(\sigma)$.   For every cell $e$ of $X$, there is a unique $\sigma$ so that $e \subset (X(\sigma) - \partial X(\sigma))$.
Then $C_nX$ is $K$-based with 
$$C_nX = \bigoplus_{\sigma \in K} C_n(X(\sigma), \partial X(\sigma)).$$
In fact $C(X) \in \ch(\Z(K)\Mod)$.

Corresponding assertions hold in the dual case.   One defines $\partial X(\sigma) = \cup_{\sigma < \tau} X(\tau)$.  Then $C_nX$ is $K$-based with 
$$C_nX = \bigoplus_{\sigma \in K} C_n(X(\sigma), \partial X(\sigma)).$$
In fact $C(X) \in \ch(\Z(K^{\op})\Mod)$.
\begin{example}

In Figure \ref{fig:space-dissections}, we give examples of a $K$-dissection and a $K^{\op}$-dissection.

   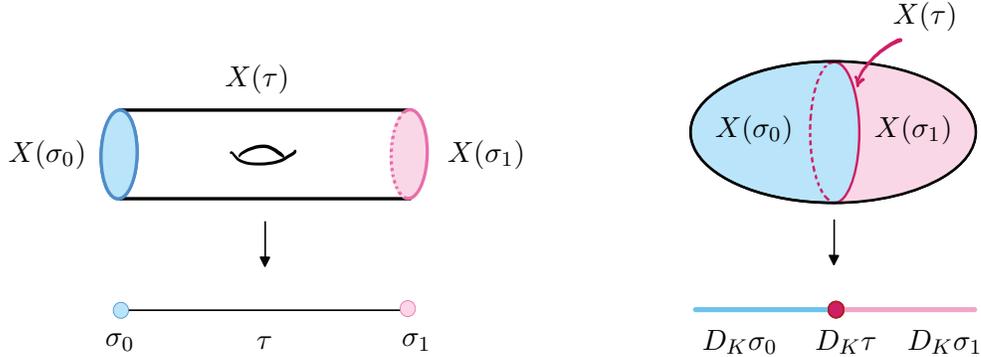
\begin{figure}[ht!] 
\labellist
\small\hair 2pt
\pinlabel \small {$\tau$} at 73 -5
\pinlabel \small {$\sigma_0$} at 13 -5
\pinlabel \small {$\sigma_1$} at 136 -5
\pinlabel \small {$X(\tau)$} at 70 100
\pinlabel \small {\small{$X(\sigma_0)$}} at -17 70
\pinlabel \small {\small{$X(\sigma_1)$}} at 165 70
\pinlabel \small {$D_K\tau$} at 314 -5
\pinlabel \small {$D_K\sigma_0$} at 270 -5
\pinlabel \small {$D_K\sigma_1$} at 355 -5
\pinlabel \small {$X(\tau)$} at 347 125
\pinlabel \small {\small{$X(\sigma_0)$}} at 276 80
\pinlabel \small {\small{$X(\sigma_1)$}} at 342 80
\endlabellist
\centering
\includegraphics[width=120mm, height= 40mm]{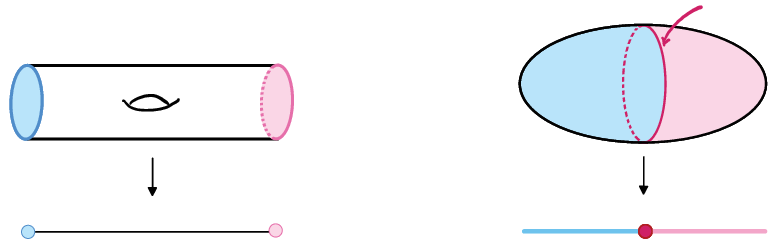}
\vspace{10pt}
\caption{$K$- and $K^{\op}$-dissections}
\label{fig:space-dissections}
\end{figure}
Note that if $X$ is an $n$-dimensional manifold, then a map $f: X \to K'$ can be made transverse to the dual cells $D_K \sigma$ for all $\sigma \in K$ so that $f^{-1}(D_K \sigma)$ is a submanifold of dimension $n- |\sigma|$.
   
\end{example}
The notions of $K$ and $K^{\op}$-dissections are special cases of the notion of a free $\cC$-CW-complex for a category $\cC$ defined in Section 3 of  \cite{DL98}.

\section{The categorical point of view} \label{sec:cat_point_of_view}

$\Z(K)\Mod$ is an additive category, but to use homological algebra, one needs to embed it in an abelian category.   In this section we develop this point of view.

Let $K$ be a small category.

Recall that $\sigma \in K$ means $\sigma$ is an object of $K$  and that $K(\sigma, \tau)$ is the set of morphisms from $\sigma$ to $\tau$.
  Given a morphism $\alpha : \sigma \to \tau$, let $s_\alpha = \sigma$, the source, and let $t_\alpha  = \tau$, the target.  Let $\mor_K$ be the set of all morphisms in $K$.

Let $\Z [K]\Mod$ be the category whose objects are functors $M: K \to \Z\Mod$ and whose morphisms are natural transformations.   An example of a $\Z[K]$-module is the trivial module $\ul\Z^K$ where $\ul\Z^K(\sigma) = \Z$ for every object $\sigma$ and $\ul\Z^K(\alpha) = \id_\Z$ for every morphism $\alpha$. The category $\Z [K]\Mod$ is an abelian category.  Its morphism sets are abelian groups.  A $\Z [K]$-module $M$ is {\em finitely generated} if there are $a_1 \in M(\sigma_1), \dots, a_k \in M(\sigma_k)$ so that for any $b \in M(\sigma)$, then $b = \sum M(\alpha _i)(a_i)$ for some morphisms $\alpha _1 \in K(\sigma_1,\sigma), \dots, \alpha _k \in K(\sigma_k,\sigma)$.  A $\Z [K]$-module is {\em free} if it is isomorphic to a direct sum of modules of the form $\Z[K(\sigma,-)]$ for some $\sigma \in K$ (the $\sigma$ can vary and repeat in the direct sum).  A {\em basis} for a $\Z[K]$-module $M$ is a collection of subsets ${B_\sigma \subset M(\sigma)}_{\sigma \in K}$ so that  any $b \in M(\tau)$ for any $\tau$ can be expressed uniquely as $b = \sum M(\alpha _i)(b_i)$ for morphisms $\alpha _i$ and basis elements $b_i$.

We write $\Hom_{\Z[K]}(M,N)$ for the morphism set $\Z[K]\Mod(M,N)$ given the structure of an abelian group as a subgroup of $\prod_{\sigma \in K} \Hom_\Z(M(\sigma), N(\sigma))$.   When $M$ is a $\Z[K^{\op}]$-module and $N$ is a $\Z[K]$-module, define the {\em tensor product} 
$$
M \otimes_{\Z[K]} N = \frac{\bigoplus_{\sigma \in K} M(\sigma) \otimes_\Z N(\sigma)}{(mf,n) \sim (m,fn)}
$$
This is an abelian group.   For an abelian group $A$,  there is the adjoint isomorphism
$$
\Hom_\Z(M \otimes_{\Z[K]} N , A) \cong \Hom_{\Z[K^{\op}]}(M, \Hom_\Z(N,A)),
$$
natural in $M$, $N$, and $A$.

Yoneda's Lemma gives isomorphisms of abelian groups
\begin{align*}
 \Hom_{\Z[K]}(\Z[K(\sigma,-)],N)& \cong N(\sigma) \\
M \otimes_{\Z[K]}  \Z[K(\sigma,-)] & \cong M(\sigma).
\end{align*}
Here for a set $S$, $\Z[S]$ is the free abelian group with basis $S$ (elements are $\sum n_is_i$) and $ \Z[K(\sigma,-)] : K \to \Z\Mod$ is the functor $\tau  \mapsto \Z[K(\sigma,\tau)]$.

Let $F : J \to K$ be a functor.  Let $F^* : \Z [K]\Mod \to \Z [J]\Mod$ be the {\em restriction} $F^*N = N \circ F$.  It has a left adjoint {\em induction} $F_* : \Z [J]\Mod \to \Z [K]\Mod$ with
$$F_*M(??) = \Z[K(F(?),??)] \otimes_{\Z [J]} M(?)$$ 
where $F^*$ is exact, and $F_*$ takes projective objects to projective objects.  They satisfy the adjoint property  (see Lemma 1.9 of \cite{DL98}).

$$
\Hom_{\Z [K]}(F_*M,N)  \xrightarrow{\cong} \Hom_{\Z [J]}(M,F^*N).
$$
In the special case $c : K \to 1$ where $1$ is the trivial category, $c_*M = \colim_K M = \ul{\Z}^{K^{\op}} \otimes_{\Z[K]} M$.

We now generalize Definition \ref{finite poset}.

\begin{definition} 
Let $K$ be a small category.
Let $\Z(K)\Mod$ be the category whose objects are $K$-based abelian groups  
$$
M = \bigoplus_{\sigma \in K} M(\sigma)
$$ 
with $M$ a finitely generated free abelian group
and whose morphisms $f: M \to N$ are a collection of homomorphisms of abelian groups $\{f_\alpha : M(t_\alpha) \to N(s_\alpha)\}_{ \alpha \in \mor_K}$, which are zero for all but a finite number of $\alpha$.  The composition law is given by
$$
(g \circ f)_\alpha = \sum_{ \beta \circ \gamma  = \alpha}  g_\gamma \circ f_\beta.
$$
\end{definition}

An example of a $\Z(K)$-module is given by $\Z_{\tau}$ where $$\Z_\tau(\sigma) = 
\begin{cases}
\Z &   \sigma = \tau \\
0 &  \sigma \not = \tau.
\end{cases}$$   Every $\Z(K)$-module is isomorphic to a finite direct sum of such modules.

Define a functor
$$
[~]: \Z(K)\Mod \to \Z[K]\Mod
$$
by sending $M$ to $[M](\sigma) = \bigoplus_{? \to \sigma} M(?)$ and $f = \{f_\alpha\}: M \to N$ to the map $[f] : [M] \to [N]$ where the  $?? \to \sigma, ? \to \sigma$ component of
$$[f](\sigma) :[M](\sigma) =  \bigoplus_{? \to \sigma} M(?)  \to [N](\sigma) = \bigoplus_{?? \to \sigma} N(??)$$
is given by the sum of all $f_{?? \to ?}$ where  $?? \to ?$ satisfies
 $?? \to \sigma = (? \to \sigma) \circ (?? \to ?)$.  
An equivalent definition of  $[~]$ is
$$
[M] = \left[  \bigoplus_{\sigma \in K} M(\sigma) \right] = \bigoplus_{\sigma \in K} \Z[K(\sigma,-)] \otimes_\Z M(\sigma).
$$
The functor $[~]$ is a full, additive embedding.  Note that there is a canonical identification $[\Z_\tau] = \Z[K(\tau, - )]$.

Note that there is an embedding $M(\sigma) \subset [M](\sigma)$, corresponding to the identity summand.   $[M]$ is a finitely generated free $\Z[K]$-module; if $\{B_\sigma\}_{\sigma \in K}$ is a collection of bases for the abelian groups $M(\sigma)$, then $\{B_\sigma\}_{\sigma \in K}$ is also a basis for the $\Z[K]$-module $[M]$.

\begin{remark} \label{R_notation}
In the case where $K$ is the category of simplices of a simplicial complex, here is the comparison of our notation with that of \cite{bluebook}:
\begin{align*}
\Z[K]\Mod &=  \A(\Z)^*[K]  \\
\Z[K^{\op}]\Mod &=  \A(\Z)_*[K] = [\Z,K]\Mod\\
\Z(K)\Mod &=  \A(\Z)^*(K)\\
\Z(K^{\op})\Mod &=  \A(\Z)_*(K) =  (\Z,K)\Mod\
\end{align*}
\end{remark}

\begin{lemma} \label{round_induction}
Let $F : J \to K$ be a functor.   Then the functor $F_* : \Z[J]\Mod \to \Z[K]\Mod$ restricts to the functor  $F_* : \Z(J)\Mod \to \Z(K)\Mod$ with 
\begin{align*}
F_*M(\tau) &= \bigoplus_{\sigma \in F^{-1} \tau} M(\sigma) \\
F_* (f)_\beta &= \sum_{\alpha \in F^{-1}\beta} f_\alpha \end{align*}
\end{lemma}

\begin{proof}
It suffices to verify this for the module $M = \Z_\tau$ in which case it follows from Yoneda's Lemma.
\end{proof}

Warning: the functor $F^* : \Z[K]\Mod \to \Z[J]\Mod$ does not restrict to a functor $ \Z(K)\Mod \to \Z(J)\Mod$.

Applying the above lemma to the constant functor $K \to 1$ gives the following corollary.

\begin{corollary} \label{colimit} 
The functor $\colim_K [\quad] : \Z(K)\Mod \to \Z\Mod$ satisfies $\colim_K [M] = \oplus M(\sigma)$ and $\colim_K [f] = \sum f_\alpha$.  In other words the colimit forgets the grading.
\end{corollary}

\begin{example} \label{augmentation}
 Let $K$ be a finite simplicial complex.    Let $\{X(\sigma) \mid \sigma \in K\}$ be a $K$-dissection  (or $K^{\op}$-dissection)  of a CW complex $X$.  Then $\colim_K [C(X)] = C(|X|)$, the usual cellular chain complex of the underlying CW-complex.       And $[C(X)](\sigma) = C(X(\sigma))$, the cellular chain complex of $X(\sigma) \subset X$.    A $\Z[K]$-basis for $C(X)$ is given by the (oriented) cells of $X$.

In particular, the augmentation
$$
\varepsilon: [\Delta K] \to \ul \Z^K
$$
gives a chain homotopy equivalence, where $\Delta K$ is the $\Z(K)$-chain complex of the simplicial $K$-dissection of $|K|$ and $\ul \Z^K$ is the $\Z[K]$-chain complex given by placing $\Z$ in degree 0 at each object.   Here $\varepsilon(\sigma) = 1$ for all 0-simplices of $K$.  

Thus $[\Delta K]$ should be regarded as a finite free $\Z[K]$-resolution of $\ul \Z^K$.
\end{example}

\begin{definition}
 Let $\cF(\Z[K]\Mod)$ be the full subcategory of $\Z[K]\Mod$ consisting of all modules $M$ so that $\oplus M(\sigma)$ is a finitely generated free abelian group.
\end{definition}

\begin{definition}[Round-square tensor products and Homs]  There are additive functors
\begin{align*}
- \otimes_{\Z(K]} -& : \Z(K^{\op})\Mod \times\ \cF(\Z[K]\Mod) \to \Z(K^{\op})\Mod \\
\Hom_{\Z(K]} (-,-) &:(\Z(K)\Mod)^{\op} \times\ \cF(\Z[K]\Mod) \to \Z(K^{\op})\Mod
\end{align*}
with
\begin{align*}
(M \otimes_{\Z(K]} N)(\sigma) & = M(\sigma) \otimes_\Z N(\sigma)  \\
\Hom_{\Z(K]} (M,N)(\sigma) & = \Hom_{\Z} (M(\sigma),N(\sigma))
\end{align*}
\begin{align*}
(f \otimes g)_{\sigma \to \tau} & = f_{\sigma \to \tau}  \otimes g(\sigma \to \tau) \\
\Hom(f, g)_{\sigma \to \tau} & = \Hom( f_{\sigma \to \tau}, g(\sigma \to \tau)) \\
\end{align*}
where $f: M \to M'$ and $g : N \to N'$ are module morphisms and $\sigma \to \tau$ is a morphism in $K$.
 \end{definition}

 By additivity and the Tot construction, these functors extend to functors on chain complexes: $C \otimes_{\Z(K]} D$ and $\Hom_{\Z(K]} (C,D)$.  A careful look at the definitions shows that $C \otimes_{\Z(K]} D$ is a subcomplex of the chain complex  $C \otimes_\Z D$ 
  Also $\Hom_{\Z(K]} (C,D)$ is a subcomplex of  $\Hom_{\Z} (C,D)$. 
 
 \begin{remark}
 We will be loose with the $\Z(K]$-decorations in the following sense.   If it is clearly labelled what $M$ and $N$ are, we will just write $M \otimes N$ or $\Hom(M,N)$.  Conversely, if we write, for example, $M \otimes_{\Z(K] } N$, then we will assume that $M$ is a $\Z(K^{\op})$-module and $N$ is a $\Z[K]$-module.
\end{remark}

\begin{lemma}\label{rules} Hom and tensor products satisfy a myriad of identities; here are some we need:
 \begin{enumerate}
 \item $\colim_{K^{\op}} [C \otimes_{\Z(K]}  D]  = \bigoplus_{\sigma \in K} C(\sigma) \otimes_{\Z} D(\sigma) = [C] \otimes_{\Z[K]}  D $.
  \item $\colim_{K^{\op}}  [\Hom_{\Z(K]}(C, D)] = \Hom_{\Z[K]}([C], D) $.
  \item $(C \otimes_{\Z(K]} [D]) \otimes_{\Z(K]} [E]  \xrightarrow{\cong} (C \otimes_{\Z(K]} [E]) \otimes_{\Z(K]} [D]~;~  \\ 
  x\otimes y \otimes z  \mapsto (-1)^{|y||z|}x \otimes z \otimes y$ 
    \item $\colim_{K^{\op}} [C^{-*}] = (\colim_K [C])^{-*}$.
    \item $C^{-*} \otimes_{\Z(K]} D^{-*}  \xrightarrow{\cong}  (C \otimes_{\Z(K^{\op}]}  D)^{-*} ~;~    \alpha \otimes \beta \mapsto (x \otimes y \mapsto (-1)^{|x||\beta|} \alpha(x) \otimes \beta(y)$
  \item $C \xrightarrow{\cong} (C^{-*})^{-*} ~;~  x \mapsto (f \mapsto (-1)^{|x|}f(x))$  for $C \in \ch(\Z(K)\Mod)$.
  \item $  C \otimes_{\Z(K]} D \xrightarrow{\cong} \Hom_{\Z(K]}(C^{-*},D)  ~;~  (x \otimes y) \mapsto (f \mapsto (-1)^{|x||y|+ |x|}f(x) y)$.
  \item $  C^{-*} \otimes_{\Z(K]} D \xrightarrow{\cong} \Hom_{\Z(K]}(C,D)  ~;~  (f \otimes y) \mapsto (x \mapsto (-1)^{|x||y|}f(x) y)$.
\end{enumerate}
\end{lemma}

For example, we will prove part 1 for modules; then the verification for chain complexes will be clear.   Let $C = \bigoplus_{\sigma \in K} C(\sigma)$ be a $K$-based abelian group and let $D$ be  
a $\Z[K]$-module.  Then
\begin{align*}
\colim_{K^{\op}} [C \otimes_{\Z(K]}  D]  & = \ul \Z^K \otimes_{\Z[K^{\op}]} \left[  C \otimes_{\Z(K]}  D \right] \\
& = \bigoplus_{\sigma \in K} \left(  \ul \Z^K \otimes_{\Z[K^{\op}]} \left(  \Z[K(-,\sigma)] \otimes_\Z C(\sigma) \otimes_\Z D(\sigma) \right) \right)\\
&= \bigoplus_{\sigma \in K}  C(\sigma) \otimes_\Z D(\sigma) \\
&=  \bigoplus_{\sigma \in K} C(\sigma) \otimes_\Z  \left( \Z[K(-,\sigma)]  \otimes_{\Z[K]} D \right) \\
& = [C] \otimes_{\Z[K]} D,
\end{align*}
where the second and fourth equalities are due to Yoneda's Lemma.

\begin{remark}
 The last three isomorphisms depend on the hypothesis that if $M\in \Z(K)\Mod$ then the underlying module is finitely generated free.    Without that hypothesis then one only gets a map from left to right.   To remind the reader of this fact, we will often write $C \xrightarrow{\cong}  (C^{-*})^{-*}$ and   $C \otimes_{\Z(K]} D \xrightarrow{\cong} \Hom_{\Z(K]}(C^{-*},D)$.  The last two isomorphisms are called {\em slant products}.
\end{remark}

\begin{definition}
A chain map $f : C \to D$ of $\Z[K]$-chain complex is a {\em weak equivalence} if for all $\sigma \in K$, $f(\sigma) : C(\sigma) \to D(\sigma)$ is an isomorphism on homology.
\end{definition}

\begin{proposition} \label{whe=che}
Let $f : A \to B$ be a map in $ \ch(\Z(K)\Mod)$.    Then $[f] : [A] \to [B]$ is a weak equivalence if and only if $f$ is a chain homotopy equivalence.   
\end{proposition}

\begin{proof}
 Standard techniques from homological algebra show that, in an abelian category if the chain complexes are projective 
and bounded below, then a weak equivalence is a chain homotopy equivalence. 
 Thus $[f]$ is a weak equivalence if and only if $[f]$ is a chain homotopy equivalence.   But 
 since the embedding of $\Z(K)\Mod$ in $\Z[K]\Mod$ is full, this occurs if and only if $f$ is a chain homotopy equivalence.  
\end{proof}

\begin{proposition} \label{whe}
Let $C, D \in \ch(\cF(\Z[K^{\op}]\Mod))$ and $E \in \ch(\Z(K)\Mod)$.   If $f : C \to D$ is a weak equivalence, then $f \otimes_{\Z[K]} \id : C \otimes_{\Z[K]} [E] \to D \otimes_{\Z[K]} [E]$ is a weak equivalence.
\end{proposition}

This is proved by induction on the length of $E$.  

The main result of this paper asserts that the additive categories of chain complexes $\ch((\Z(K^{\op})\Mod)$ and $\ch((\Z(K)\Mod)$ admit a chain duality for a finite simplicial complex $K$.    The round-square tensor product is enough to construct the chain duality functor, but to prove that it is a chain duality we need to introduce the round tensor product and the category below.   We discuss the geometric motivation in the next section.  

For a simplicial complex $K$, let $DK$ be the poset whose objects are 1-chains 
$(\sigma \leq \tau)$
and whose inequalities are $(\sigma \leq \tau) \leq (\sigma' \leq \tau')$ if and only if $\sigma' \leq \sigma \leq \tau \leq \tau'$.  (A 1-chain is simply a pair of simplices $\sigma,\tau \in K$ with $\sigma$ a face of $\tau$.)   There are morphisms of posets 
$$
\begin{tikzcd}
& K^{\op}\\
DK \arrow[ru,"\pi^K"] \arrow[rd,"\pi_K"] & \\
& K
\end{tikzcd}
$$
with $\pi^K(\sigma \leq \tau)  = \sigma$ and $\pi_K(\sigma \leq \tau)  = \tau$.   Note that $DK \cong DK^{\op}$ with 
$(\sigma \leq \tau) \leftrightarrow (\tau \geq\sigma)$.

In figure \ref{fig:dissection-2-simplex}, we see an example of the geometric realization of the posets $DK$ (center of the figure), $K$, and $K^{\op}$ (lower and upper right side respectively) when $K$ is a $2$-simplex.

\begin{definition}[Round tensor products and Homs]  We have the following additive functors
\begin{enumerate}
\item over $\Z(K)$,
  \begin{align*}
 - \otimes_{\Z(K)}- & :  \Z(K^{\op})\Mod \times \Z(K)\Mod \to \Z(DK)
\Mod \\
 \Hom_{\Z(K)}(-,-)  &: (\Z(K)\Mod)^{\op} \times~ \Z(K)\Mod \to \Z(DK )\Mod
\end{align*}
with 
\begin{align*}
 (M \otimes_{\Z(K)} N)(\sigma \leq \tau) &= M(\sigma) \otimes_\Z N(\tau)\\
 \Hom_{\Z(K)}(M,N)(\sigma \leq \tau)  &= \Hom_{\Z}(M(\sigma),N(\tau))
 \end{align*}
\begin{align*}
(f \otimes g)_{(\sigma \leq \tau) \leq (\sigma' \leq \tau')}
 & = f_{\sigma' \leq \sigma}  \otimes g_{\tau \leq \tau'} \\
\Hom(f, g)_{(\sigma \leq \tau) \leq (\sigma' \leq \tau')}
& = \Hom( f_{\sigma' \leq \sigma}, g_{\tau \leq \tau'}) \\
\end{align*}

\item over $\Z(K^{\op})$,
\begin{align*}
 - \otimes_{\Z(K^{\op})}- & :  \Z(K)\Mod \times \Z(K^{\op})\Mod \to \Z(DK^{\op})
\Mod \\
 \Hom_{\Z(K^{\op})}(-,-)  &: (\Z(K^{\op})\Mod)^{\op} \times~ \Z(K^{\op})\Mod \to \Z(DK^{\op} )\Mod
\end{align*}
with 
\begin{align*}
 (M \otimes_{\Z(K^{\op})} N)(\sigma \leq \tau) &= M(\tau) \otimes_\Z N(\sigma)\\
 \Hom_{\Z(K^{\op})}(M,N)(\sigma \leq \tau)  &= \Hom_{\Z}(M(\tau),N(\sigma))
 \end{align*}
\begin{align*}
(f \otimes g)_{(\sigma\leq \tau) \leq (\sigma' \leq \tau')} 
& = f_{\tau \leq \tau'}  \otimes g_{\sigma' \leq \sigma} \\
\Hom(f, g)_{(\sigma \leq \tau) \leq (\sigma'\leq  \tau')}
 & = \Hom( f_{\tau \leq \tau'}, g_{\sigma' \leq \sigma}) \\
\end{align*}

\end{enumerate}
\end{definition}

 By additivity and the Tot construction, these functors extend to functors on chain complexes: $\begin{cases} C \otimes_{\Z(K)} D \\ C \otimes_{\Z(K^{\op})} D \end{cases}$  and  $\begin{cases} \Hom_{\Z(K)} (C,D)  \\ \Hom_{\Z(K)^{\op}} (C,D) \end{cases}.$
 

\medskip

In the formula for $e_C : T^2C \to C$, the $\Z(DK)$-chain complex $\Delta K^{-*} \otimes_{\Z(K)} \Delta K$ arises.    We need a weak equivalence
$$
\varepsilon: [\Delta K^{-*} \otimes_{\Z(K)} \Delta K] \to \ul\Z^{DK}
$$
similar to the augmentation of Example \ref{augmentation}.

The (oriented) simplices $\{\sigma\}$ give a basis for $\Delta K$ and the dual basis $\{\wh \sigma\}$ gives a basis for $\Delta K^{-*}$.   A basis of $\Delta K^{-*} \otimes_{\Z(K)} \Delta K$ is given by $\{\wh \sigma \otimes \tau\}$, where $\sigma$ is a face of $\tau$.
  Define  $\varepsilon(\wh \sigma \otimes \sigma) = 1$ for any simplex $\sigma$ and $\varepsilon(\wh \sigma \otimes \tau) = 0$ if $\sigma$ is a proper face of $\tau$.
To verify that $\varepsilon$ is a chain map, we need to show that  if  $\dim \tau = \dim \sigma + 1$, then $\varepsilon(\partial(\wh \sigma \otimes \tau)) =  \partial(\varepsilon(\wh \sigma \otimes \tau))  = \partial(0) = 0$.  Let $[\tau:\sigma]$ be the incidence number: the coefficient of $\sigma$ in $\partial \tau$.     Note
$$
\partial (\wh \sigma \otimes \tau) = [\tau: \sigma]((-1)^{|\sigma|+1} \wh \tau \otimes \tau + (-1)^{|\sigma|} \wh \sigma \otimes \sigma).
$$
Hence 
$$
\varepsilon ( \partial (\wh \sigma \otimes \tau) ) = 0.
$$

Thus $\varepsilon$ is a chain map.   We provide a proof that  $\varepsilon$ is a chain equivalence with the help of the geometry of the dual cell decomposition in the next section.  
\medskip

Note that similarly, a $\Z[DK^{\op}]$-basis of $[\Delta K \otimes_{\Z(K^{\op})} \Delta K^{-*}] $ is given by $ \tau \otimes \wh \sigma$, where $\sigma$ is a face of $\tau$.
  We define  $\varepsilon(\sigma \otimes \wh \sigma) = (-1)^{|\sigma|}$ for any simplex $\sigma$ and $\varepsilon(\tau \otimes \wh \sigma) = 0$ if $\sigma$ is a proper face of $\tau$.
To verify that $\varepsilon$ is a chain map, we need to show that  if  $\dim \tau = \dim \sigma + 1$, then $\varepsilon(\partial(\tau \otimes \wh \sigma)) =  \partial(\varepsilon(\tau \otimes \wh \sigma))  = \partial(0) = 0$.  Like before, let $[\tau:\sigma]$ be the incidence number.     Note
\begin{align*}
\partial (\tau \otimes \wh \sigma) &  = [\tau: \sigma]( \sigma \otimes  \wh \sigma + (-1)^{|\tau|} (-1)^{|\sigma|+1}  \tau \otimes \wh \tau) \\
& = [\tau: \sigma]( \sigma \otimes  \wh \sigma +   \tau \otimes \wh \tau).
\end{align*}
Hence, 
$$
\varepsilon ( \partial (\tau \otimes \wh \sigma) ) = 0.
$$
Thus,
$$
\varepsilon: [\Delta K \otimes_{\Z(K^{\op})} \Delta K^{-*}] \to \ul\Z^{DK^{\op}}
$$
is a chain map.

\section{The dual cell decomposition} \label{sec:dual_cell}

A regular CW-complex is a CW-complex where the closure of each open $n$-cell is homeomorphic to an $n$-disk.   We denote a regular CW-complex $X$ by a pair $(|X|, \{\sigma^n\})$ where $|X|$ is a topological space and each $\sigma^n$ is a (closed) $n$-cell in $|X|$, and   $(|X|, \{\sigma^n\})$ is a subdivision of $(|X|, \{\tau^n\})$ if every cell of the former is a subset of a cell of the latter:   $\sigma^n \subset \tau^n$.   

Regular CW-complexes are closely related to simplicial complexes.   A simplicial complex is a regular CW-complex (with simplices as cells) and every regular CW-complex has a simplicial subdivision.

Let $K = (|K|, \{\sigma\})$ be a simplicial complex.   We will define the dual cell subdivision $DK = (|K|, \{D_\tau \sigma\})$;  its poset of cells will be the category $DK$ mentioned earlier.   The dual cell decomposition  is a regular CW-complex intermediary between $K$ and its barycentric subdivision $K'$.    The utility of the dual cell decomposition is that it is useful for the proof that the category $\ch(\Z (K^{\op})\Mod)$ admits a  chain duality.

We will first remind the reader of the barycentric subdivision $$K' = (|K|, \{ \ol \sigma_0 \ol \sigma_1 \dots \ol \sigma_n\}).$$    Here the $\sigma_i$ are simplices of $K$ satisfying $\sigma_0 < \sigma_1 < \dots < \sigma_n$, $\ol \sigma$ is the barycenter of $\sigma$, and $ \ol \sigma_0 \ol \sigma_1 \dots \ol \sigma_n$ is the convex hull of these barycenters.   

For $\sigma \leq \tau \in K$, the dual cell $D_{\tau} \sigma$ is the union of all simplices $\ol{\sigma}_0 \ol{\sigma}_1 \dots \ol{\sigma}_r$ of the barycentric subdivision so that $\sigma \leq \sigma_0 < \sigma_1 < \dots < \sigma_r \leq \tau$.

\begin{lemma} \label{dual cells are cells}
For $\sigma \leq \tau \in K$, the dual cell $D_{\tau} \sigma \subset |K|$ is homeomorphic to a disk $D^{\dim \tau - \dim \sigma}$.
\end{lemma}

\begin{proof}
Please refer to the Figure \ref{fig:radial}  while reading the proof.    Here  $\tau$ is a 2-simplex and $\sigma$ is the lower left vertex.
\begin{figure}[ht!] 
\centering
 \begin{tikzpicture}[x = 1cm, y = 1cm]
\draw (0,0) -- (60:4);
\draw[blue] (60:4) -- (4,0);
\draw(4,0) -- (0,0);
\draw[red] (60:2) -- (2,{(2/3)*sqrt(3)});
\draw[red] (2,0) -- (2,{(2/3)*sqrt(3)});
\node (a) at (0,0) {$\bullet$};
\node (b) at  (1,.5) {$D_\tau \sigma$};
\node (c) at (-.3,-.3) {$\ol\sigma=\sigma$};
\node (d) at (2,2.4) {$\tau$};
\end{tikzpicture}
\caption{$\sigma_C$ (blue) and $L_{\tau} \sigma$ (red) }
\label{fig:radial}
\end{figure}
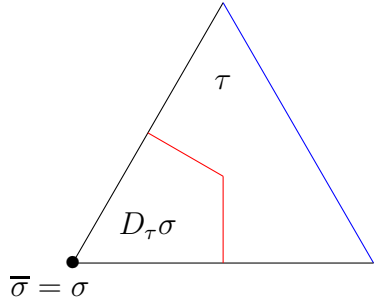

Let $\sigma_C$ be the subset of $\tau$ spanned by the vertices of $\tau$ which are not vertices of $\sigma$ (in the example above $\sigma_C$ is the blue line segment).   Then $\sigma_C$ is a simplex of dimension $\dim \tau - \dim \sigma - 1$.   Let $\ol \sigma * \sigma_C$ be the union of all line segments between $\ol \sigma$ and a point in $\sigma_C$.    Then $\ol \sigma * \sigma_C$ is homeomorphic to the cone on $\sigma_C$, and hence is homeomorphic to $D^{\dim \tau - \dim \sigma}$.

  Let $L_\tau \sigma$ be the union of all simplices $\ol \sigma_0 \ol \sigma_1 \dots \ol \sigma_r$ of the barycentric subdivision so that $\sigma < \sigma_0 < \sigma_1 < \dots < \sigma_r \leq \tau$ (in the example above $L_\tau \sigma$ is in red).

A radial argument using line segments starting at $\ol \sigma$ shows that  $D_\tau \sigma=\ol \sigma *L_\tau \sigma$ is homeomorphic to  $\ol \sigma * \sigma_C$.   

Here are some details.   Embed $\tau$ isometrically into $\R^n$ with $\ol \sigma$ mapping to the origin.   
For $L \subset \R^n$, the {\em cone on $L$} is $$\cone(L) = \{ t\bfx \mid t \in [0,1], \bfx \in L\},$$   where $\cone(L)$ is a {\em proper cone}  if $L$ is compact, ${\bf 0} \not \in L$, and  every ray starting at the origin intersects $L$ in at most one point.    Let $S(L) = \{\bfx/\|\bfx\| \mid \bfx \in L\}$.  Two proper cones $\cone(L)$ and $\cone(M)$ are {\em commensurate} if every ray starting at the origin intersects $L$ if and only if it intersects $M$, equivalently, $S(L) = S(M)$.  As an example, If $\sigma$ is placed at the origin, $D_\tau \sigma = \cone(L_\tau \sigma)$ and $\cone(\sigma_C)$ are commensurate.

We claim that commensurate cones are homeomorphic.   Using compactness of $L$ one shows that  map $L \to S(L), \bfx \mapsto \bfx/\|\bfx\|$ is a homeomorphism, that the maps $L \times [0,1] \to \cone(L), (\bfx,t) \mapsto t\bfx$ and 
$L \times [0,1] \to \cone(S(L)), (\bfx,t) \mapsto t\bfx/\|\bfx\|$ are quotient maps, and that the induced map $\cone(L) \to \cone(S(L))$ is a homeomorphism.   
\end{proof}

\begin{lemma}
If $K$ is a finite simplicial complex, then $\left(|K|, \{ D_\tau \sigma \}_{\sigma \leq \tau \in K} \right)$ is a regular CW-complex.
\end{lemma}

\begin{proof}
A finite regular CW-complex is a compact Hausdorff space covered by a finite collection of closed cells, whose interiors partition the space, and so that the boundary of each cell is a union of closed cells.    These conditions are satisfied, since
$$
\partial D_\tau \sigma = \left(\bigcup_{\sigma < \sigma' \leq \tau}  D_\tau \sigma' \right) \cup  \left( \bigcup_{\sigma \leq \tau' < \tau}  D_{\tau'} \sigma \right).
$$
\end{proof}

  For a regular CW-complex $Z$, one can define a $Z$-dissection and a $Z^{\op}$-dissection, as in Definition \ref{def_dissection}.    Similar to Example \ref{K_dissected_by_K}, the regular CW-complex $DK$ has a $DK$-dissection.   Thus the chain complex $C(DK)$ is a $\Z(DK)$-chain complex.

An orientation on an $n$-cell $\sigma$ of a regular $CW$-complex is a choice of generator of the infinite cyclic group $H_n(\sigma, \partial \sigma)$.    We  require that all 0-cells are positively oriented, in other words the generator is the unique singular 0-chain with image the 0-cell.  An oriented regular $CW$-complex is a CW-complex with an orientation for each cell.   We will abuse notation and, in an oriented CW-complex $X$, use $\sigma$ to denote an $n$-cell, the generator of  $H_n(\sigma, \partial \sigma)$, and the image element in $C_n(X)$. The cellular chains of an oriented regular $CW$-complex are based.   In a regular $CW$-complex we will use the notation $\sigma' \prec \sigma$ to mean that $\sigma'$ is a codimension one face of   $\sigma$, that is, $\sigma'$ and $\sigma$ are both cells, $\sigma' \subset \partial \sigma$, and $1 + \dim \sigma' = \dim \sigma$.  If $(X, \{\sigma^n\})$ is a oriented regular $CW$-complex, and, in the cellular chain complex $\partial \sigma = \sum_{\sigma' \prec \sigma} [\sigma: \sigma'] \sigma'$, then the coefficients $[\sigma : \sigma']$ are called incidence numbers; note that $[\sigma: \sigma'] = \pm 1$.

We wish to orient the dual cell decomposition $DK$ of $|K|$ with an orientation which only depends on the orientation of $K$.   To quote William Massey  (see page 243 of \cite{Ma91}): ``we can specify orientations for the cells of a regular $CW$-complex by specifying a set of incidence numbers for the complex.   This is one of the most convenient ways of specifying orientations of cells."
The following theorem is essentially Theorem 7.2, Chapter IX of \cite{Ma91}:

\begin{theorem} \label{Massey}
 Let $X$ be a regular CW-complex.   For each pair $(\tau^n,\sigma^{n-1})$ consisting of an $n$-cell and an $(n-1)$-cell of $X$, let there be given an integer $\alpha^n_{\tau\sigma} = 0$ or $\pm 1$ such that the following four conditions hold:
 \begin{enumerate}[(1)]
\item If $\sigma^{n-1}$ is not a face of $\tau^n$, then $\alpha^n_{\tau\sigma} = 0$.
\item If $\sigma^{n-1}$ is  a face of $\tau^n$, then $\alpha^n_{\tau\sigma} = \pm 1$.
\item If $\sigma^0$ and $\rho^0$ are the two vertices of the 1-cell $\tau^1$, then
$$
\alpha^1_{\tau\sigma} + \alpha^1_{\tau\rho} = 0.
$$
\item Let $\tau^n$ and $\sigma^{n-2}$ be cells of $X$ so that $\sigma^{n-2} \subset \tau^n$; let $\rho^{n-1}$ and $\nu^{n-1}$ be the unique $(n-1)$-cells so that
$\sigma^{n-2} \subset \rho^{n-1}  \subset \tau^n$ and $\sigma^{n-2} \subset   \nu^{n-1} \subset \tau^n$. 
Then
$$
\alpha^n_{\tau\rho} \alpha^{n-1}_{\rho\sigma} + \alpha^n_{\tau\nu} \alpha^{n-1}_{\nu\sigma} =0.
$$
\end{enumerate}
Under these assumptions,  there exists a unique choice of orientations for the cells so that in the cellular chain complex $C(X)$
$$
\partial \tau^n = \sum_{\sigma^{n-1} \prec {\tau}^n} \alpha^n_{\tau\sigma}  \sigma^{n-1}.
$$
\end{theorem}

Thus the $\alpha$'s not only determine orientations on all the cells, they identify the cellular chain complex $C(X)$.

\begin{theorem}
 Let $K$ be a finite oriented simplicial complex.   There is an orientation on the regular CW-complex $DK$ and an isomorphism of $\Z(DK)$-complexes
\begin{align*}
    C(DK) & \cong \Delta K^{-*} \otimes_{\Z(K)} \Delta K  \\
        D_\tau \sigma & \mapsto \wh \sigma \otimes \tau
      \end{align*}
\end{theorem}

\begin{proof}  We have
$$
\partial (\wh \sigma \otimes \tau) =    (-1)^{|\sigma|+1} 
 \sum_{\sigma \prec \sigma' \leq \tau} [\sigma':\sigma]  \wh \sigma' \otimes \tau + (-1)^{|\sigma|} \sum_{\sigma \leq \tau' \prec \tau}  [\tau:\tau']     \wh \sigma \otimes \tau'
$$
Thus if $\sigma \prec \sigma' \leq \tau$, define $\alpha^n_{D_{\tau}\sigma D_{\tau}\sigma'} = (-1)^{|\sigma|+1} [\sigma': \sigma] $ and if $\sigma \leq \tau' \prec \tau$,  $\alpha^n_{{D_{\tau}\sigma}D_{\tau'}\sigma} = (-1)^{|\sigma|}[\tau: \tau'] $.   

Finally, we should check that the $\alpha$'s satisfy the hypothesis of Theorem \ref{Massey}.   There are two possibilities for codimension one faces of $D_{\tau} \sigma$, namely $D_{\tau} \sigma'$ where $\sigma \prec \sigma' \leq \tau$ and $D_{\tau'} \sigma$ where $\sigma \leq \tau' \prec \tau$.   In both cases, parts (1), (2), and (3) of Theorem \ref{Massey} are satisfied.    There are three possibilities for codimension 2 faces of $D_\tau \sigma$, namely  $D_{\tau} \sigma''$,  $D_{\tau''} \sigma$, and $D_{\tau'} \sigma'$, when $\sigma$ is a codimension two face of $\sigma''$,  when $\tau''$ is a codimension two face of $\tau$, and when $\sigma' \prec \sigma$ and $\tau' \prec \tau$. (4) is satisfied in each case since $\Delta K^{-*} \otimes_{\Z(K)} \Delta K $ is a chain complex.

This orients $DK$ and shows that the map is an isomorphism of chain complexes.  
\end{proof}

\begin{example} Using the assumptions in Theorem \ref{Massey}, figure \ref{fig:cell-orientations} shows the orientations of the dual cells in a 2-simplex.
      
   \begin{figure}[ht!]
\labellist
\small\hair 2pt
\endlabellist
\centering
\includegraphics[width=40mm, height= 35mm]{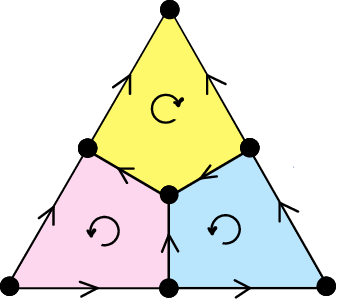}
\vspace{10pt}
\caption{Orientation of dual cells in a 2-simplex}
\label{fig:cell-orientations}
\end{figure}

\end{example}

\begin{corollary} \label{varepsilon is a weak eq}
For $\sigma \leq \tau \in K$, $H_i([\Delta K^{-*} \otimes_{\Z(K)} \Delta K])(\sigma
 \leq 
\tau)$ is zero for $i >0$ and is infinite cyclic for $i = 0$.   In fact $\varepsilon : [\Delta K^{-*} \otimes_{\Z(K)} \Delta K] \to \ul \Z^{DK}$ is a weak equivalence.
\end{corollary}

\begin{proof}
 The previous theorem shows that 
 $$[\Delta K^{-*} \otimes_{\Z(K)} \Delta K]
(\sigma \leq \tau)
 \cong C(DK)
(\sigma\leq \tau)
 = C(D_{\tau}\sigma).$$
Lemma \ref{dual cells are cells} shows that these chain complexes have the homology of a point.    Since
$\varepsilon(\sigma \leq \tau)$
 is onto the result follows.
\end{proof}

\begin{lemma}\label{iso-in-DK}  There is an isomorphism of chain complexes 
 \begin{align*}
 p : [\Delta K^{-*} \otimes_{\Z(K)} \Delta K]
(\sigma \leq \tau) 
& \to [\Delta K \otimes_{\Z(K^{\op})}  \Delta K^{-*}](
\tau \geq \sigma) \\
 \wh \sigma \otimes \tau & \mapsto   (-1)^{|\sigma||\tau|} \tau \otimes \wh \sigma.
 \end{align*}
\end{lemma}

\begin{proof}
First note that in this statement we are using the fact that $DK \cong DK^{\op}$ with 
$(\sigma \leq \tau) \leftrightarrow (\tau \geq \sigma)$.

The isomorphism as graded groups is clear. We need to check that the map $p$ commutes with differentials.
\begin{align*}
p  \partial (\wh \sigma \otimes \tau) & = (-1)^{|\sigma'||\tau|} (-1)^{|\sigma|+1} 
 \sum_{\sigma \prec \sigma' \leq \tau} [\sigma':\sigma]  \tau \otimes \wh \sigma' + (-1)^{|\sigma||\tau'|}(-1)^{|\sigma|} \sum_{\sigma \leq \tau' \prec \tau}  [\tau:\tau']     \tau' \otimes \wh \sigma \\
 & = (-1)^{|\sigma||\tau|}     \left(
 \sum_{\sigma \leq \tau' \prec \tau}  [\tau:\tau']     \tau' \otimes \wh \sigma   +     (-1)^{|\tau|} (-1)^{|\sigma|+1} \sum_{\sigma \prec \sigma' \leq \tau} [\sigma':\sigma]  \tau \otimes \wh \sigma' \right) \\
 & = \partial p (\wh \sigma \otimes \tau).
 \end{align*}

\end{proof}

\begin{corollary}   \label{varepsilon is a weak eq op}
The chain map 
$$
\varepsilon: [\Delta K \otimes_{\Z(K^{\op})} \Delta K^{-*}] \to \ul\Z^{DK^{\op}}
$$
is a weak equivalence.
\end{corollary}  

\begin{proof}
This follows from Lemma \ref{iso-in-DK}  and Corollary \ref{varepsilon is a weak eq}.
\end{proof}

\begin{remark} \label{two-sided}
 Davis and L\"uck's  \cite[ Lemma 3.17]{DL98} gives, for an arbitrary category $\cC$, a free $\cC^{\op} \times \cC$-CW-approximation of the discrete $\cC^{\op} \times \cC$-space $\cC(?,??)$.  This CW-approximation is foundational for homotopy (co)limits, and is essential for a 2-sided bar resolution.   The material in this section shows that there is a much smaller model for this CW-approximation in the case where $K$ is the poset of a finite simplicial complex, namely one can use the $K^{\op} \times K$-space given by
 $$
 (\sigma, \tau) \mapsto 
\begin{cases}
 D_\tau \sigma & \sigma \leq \tau \\
 \emptyset & \text{otherwise}
\end{cases}
 $$
\end{remark}

\section{$K$-based chain duality} \label{K-based chain duality}

Recall our main theorem: $\begin{cases} \ch((\Z(K^{\op})\Mod) \\ \ch((\Z(K)\Mod) \end{cases}$ admits the structure $(T, \tau)$ of a category with  chain duality.   Note that in what follows we will give more details for the results in the $\Z(K^{\op})\Mod$ category.  The reason for this choice is that the $\Z(K^{\op})\Mod$ category is more relevant for applications.

 We now define $T$.

\begin{definition}
 Let $K$ be a finite simplicial complex.  Define
 \begin{enumerate}
 \item  The $T$ functor for  $\ch(\Z(K^{\op})\Mod) $ as
  \begin{align*}
 T = T_K &: \ch(\Z(K^{\op})\Mod)  \to \ch(\Z(K^{\op})\Mod)^{\op} \\
 TC & = (\Delta K \otimes_{\Z(K^{\op}]} [C])^{-*}.
\end{align*}
 \item  The $T$ functor for  $\ch(\Z(K)\Mod) $ as
   \begin{align*}
 T = T_K &: \ch(\Z(K)\Mod)  \to \ch(\Z(K)\Mod)^{\op} \\
 TC & = (\Delta K^{-*} \otimes_{\Z(K]} [C])^{-*}.
\end{align*}
\end{enumerate}
\end{definition}

To give credence to the assertion that this is a form of duality, we pause for the following proposition, which asserts that $TC$ is a $K$-based dual of the result of forgetting the $K$-based structure of $C$.

\begin{proposition} The map
$$
(\varepsilon \otimes \id)^{-*} : (\ul \Z^K \otimes_{\Z[K^{\op}]} [C])^{-*} \to \colim_{K^{\op}} [TC]
$$
is a chain homotopy equivalence.    Thus the dual of the underlying chain complex of $C$ is chain homotopy equivalent to the underlying chain complex of $TC$.
\end{proposition}

\begin{proof}
The augmentation
$$
\varepsilon : [\Delta K] \to \ul \Z^K
$$
gives a weak  equivalence, and hence by Proposition \ref{whe}  
$$
\varepsilon \otimes \id : [\Delta K] \otimes_{\Z[K^{\op}]} [C]  \to \ul \Z^K \otimes_{\Z[K^{\op}]} [C]
$$
is a weak  equivalence.  The chains in both the domain and codomain are free abelian groups, so this map is a chain homotopy equivalence.   Thus so is its dual,
$$
(\varepsilon \otimes \id)^{-*} : (\Z^K \otimes_{\Z[K^{\op}]} [C])^{-*}      \to       ([\Delta K] \otimes_{\Z[K^{\op}]} [C] )^{-*}.   
$$

By Lemma \ref{rules} parts (1) and (4),
\begin{align*}  
([\Delta K] \otimes_{\Z[K^{\op}]} [C] )^{-*} & = (\colim_{K}   [  \Delta K \otimes_{\Z(K^{\op}]} [C] ])^{-*} \\
                                                                 & = \colim_{K^{\op}} [( \Delta K \otimes_{\Z(K^{\op}]} [C])^{-*}] \\
                                                                 &=\colim_{K^{\op}} [TC].
\end{align*}
  \end{proof}

\begin{proposition}  \label{cd}
For $C,D \in \begin{cases} \ch(\Z(K^{\op})\Mod) \\ \ch(\Z(K)\Mod)    \end{cases},$ there is a natural isomorphism

$$
\begin{cases}
\tau_{C,D}: \Hom_{\Z(K^{\op}]} (TC,[D])  \to \Hom_{\Z(K^{\op}]}(TD,[C]) \\
\tau_{C,D}: \Hom_{\Z(K]}(TC,[D])  \to \Hom_{\Z(K]}(TD,[C]) 
\end{cases}
$$
satisfying $\tau_{D,C} \circ \tau_{C,D} = \id$.
\end{proposition}

\begin{proof}  We give the proof for $C, D \in \ch(\Z(K^{\op})\Mod)$ as follows:
\begin{align*}
 \Hom_{\Z(K^{\op}]}(TC,[D]) & = \Hom_{\Z(K^{\op}]}((\Delta K \otimes_{\Z(K^{\op}]} [C])^{-*},[D])  \\
 & \xleftarrow{\cong} (\Delta K \otimes_{\Z(K^{\op}]}  [C]) \otimes_{\Z(K^{\op}]}  [D], \textnormal{ by Lemma \ref{rules} part (7)}  \\
& \cong(\Delta K \otimes_{\Z(K^{\op}]}  [D]) \otimes_{\Z(K^{\op}]}  [C] \\
 & \xrightarrow{\cong} \Hom_{\Z(K^{\op}]}((\Delta K \otimes_{\Z(K^{\op}]}  [D])^{-*},[C]) \\
 &= \Hom_{\Z(K^{\op}]}(TD,[C]).
\end{align*} 


%
%
%
%

The proof with $C, D \in \ch(\Z(K)\Mod)$ is similar. The only difference is to use the appropriate definition for $TC$ if $C$ is a $\Z(K)\Mod$ chain complex.

\end{proof}

By taking colimits, and restricting to 0-cycles, we obtain the following corollary.

\begin{corollary} The following is an additive category with a chain duality functor:
$$\begin{cases}  (\ch(\Z(K^{\op})\Mod),T, \tau), \\ (\ch(\Z(K)\Mod),T, \tau).   \end{cases}$$
\end{corollary}

Ranicki's book \cite{bluebook} does prove this corollary.   However, what is missing is a proof that $e_C$ is a chain homotopy equivalence, i.e.~this chain duality functor produces a chain duality.

By Lemma \ref{ec_appears}, there is a natural  $\begin{cases} \Z(K^{\op}) \\  \Z(K)\end{cases}$-module chain map $$e_C = \tau(\id_{TC}) : T^2C \to C.$$    Our next task will be to show that the chain duality functor is a  chain duality, i.e. that $e_C$ is a chain homotopy equivalence.    But first we need a formula for $e_C$.  In the following theorem we work with $\Z(K^{\op})\Mod$ chain complexes  and use the triple tensor product.   The corresponding statement for $\Z(K)\Mod$ is very similar except for some details with the signs involved in the vertical isomorphism. We give the formulas in this case in Theorem \ref{formula-eC-K} below.

To deal with $T^2C$ we have to make use of one more construction, which will allow for a convenient coordinate change.

\begin{definition} Let $K$ be a poset.
Let $C$ be a $\Z(K^{\op})$-chain complex, $D$ be a $\Z(K)$-chain complex, and $E$ be a $\Z[K^{\op}]$-chain complex.
 The {\em triple tensor product} is the $\Z (K^{op})$-chain complex 
$$
C \otimes D \otimes E = \pi_{K*} (C \otimes_{\Z (K)}  D) \otimes_{\Z(K^{\op}]} E
$$
Then
$$
(C \otimes D \otimes E)(\sigma) = \bigoplus_{\sigma \leq \rho} C(\sigma) \otimes_{\Z} D(\rho) \otimes_{\Z} E(\rho)
$$
(see Lemma \ref{round_induction} and Lemma \ref{rules}.1) 
and for $c \otimes d \otimes e \in C_i(\sigma) \otimes D_j(\rho) \otimes E_k(\rho)$, the boundary map is 
\begin{multline*}
\partial ( c \otimes d \otimes e) = \\ \sum_{\sigma \leq \sigma' \leq \rho} \left(\partial^C_{\sigma \leq \sigma'} c  \otimes d \otimes e+   (-1)^i c \otimes \partial^D_{\sigma' \leq \rho} d \otimes E(\sigma' \leq \rho)e \right)   +   (-1)^{i+j} c \otimes d \otimes \partial^E e.
\end{multline*}
\end{definition}

\begin{theorem} \label{triangle}  The following assertions are valid.
\begin{enumerate}
 \item There are an isomorphism  of $\Z[K^{\op}]$-chain complexes 
 
\begin{align*}
 \Psi:  \Delta K^{-*} \otimes \Delta K \otimes [C] & \to T^2C  =  (\Delta K \otimes  [(\Delta K \otimes [C])^{-*}] )^{-*} \\
 \end{align*}
given by 
\begin{multline*}
\Psi_n(\sigma):  \bigoplus_{\sigma \leq \rho \leq \tau}  ( \Delta K^{-*}(\sigma) \otimes \Delta K(\rho) \otimes C(\tau))_n \to ((\Delta K \otimes  [(\Delta K \otimes [C])^{-*}] )^{-*} )_n(\sigma)
\end{multline*}
where
 $$
 \Psi_n(\sigma)(\wh \sigma \otimes \rho \otimes a) =  \left(\sigma \otimes f \mapsto (-1)^{(1+|\sigma|) (n+|\sigma|)}  f(\rho \otimes a)\right).
 $$

 \item

If $\Phi_n(\sigma \leq \sigma')$ is the map defined by the commutative triangle:
$$
\begin{tikzcd}
  \bigoplus_{\sigma \leq \rho \leq \tau} ( \Delta K^{-*}(\sigma) \otimes \Delta K(\rho) \otimes C(\tau))_n  \arrow[rd,"\Phi_n(\sigma \leq \sigma')"] \arrow[dd,"\Psi_n(\sigma)" ']   &\\
 & C_n(\sigma') \\
 ( T^2C)_n(\sigma)\arrow[ru,"(e_C)_n(\sigma \leq \sigma') "'] &
\end{tikzcd}
$$

\medskip

then

\begin{itemize}
 \item $\Phi_n(\sigma \leq \sigma')|_{ \Delta K^{-*}(\sigma) \otimes \Delta K(\rho) \otimes C(\tau)} = 0$ if $\sigma <\rho$.
 \item $\Phi_n(\sigma \leq \sigma')|_{ \Delta K^{-*}(\sigma) \otimes \Delta K(\rho) \otimes C(\tau)} = 0$ if $\tau \neq \sigma '$.
 \item  $\Phi_n(\sigma \leq \sigma')(\wh \sigma \otimes \sigma \otimes a) =   a$, with $\sigma$ a generator of the infinite cyclic group $\Delta K_{|\sigma|}(\sigma)$ and  $\wh \sigma$ the dual generator of $\Delta K^{|\sigma|}(\sigma) $   and $a \in   C_n(\sigma')$.
   \end{itemize}
  
  \end{enumerate}
 
\end{theorem}

\begin{proof} 1.    Note that in the following lines $n  + |\sigma| = |a| + |\rho|$. 
We use Lemma \ref{rules} parts 6 and 5,

\begin{align*}
\bigoplus_{\sigma \leq \rho \leq \tau} \left(\Delta K^{-*}_{-|\sigma|}(\sigma) \otimes  \left(
\Delta K(\rho) \otimes C(\tau)\right)_{n+|\sigma|} \right)_n    &    ~  ;  ~   \wh \sigma \otimes \rho \otimes a \\
\xrightarrow{\cong}    \bigoplus_{\sigma \leq \rho \leq \tau} \left(\Delta K^{-*}_{-|\sigma|}(\sigma) \otimes  \left(
(\Delta K(\rho) \otimes C(\tau))^{-*}\right)^{-*}_{n+|\sigma|} \right)_n    &    ~  ;  ~  \wh \sigma \otimes (f \xmapsto{\beta} (-1)^{n+|\sigma|} f(\rho \otimes a)) \\
=   \left(\Delta K^{-*}_{-|\sigma|}(\sigma) \otimes  \left(
[TC](\sigma)\right)^{-*}_{n+|\sigma|} \right)_n    &    ~  ;  ~  \wh \sigma \otimes \beta \\
\xrightarrow{\cong}    \left(\Delta K_{|\sigma|}(\sigma) \otimes  
[TC](\sigma)_{-n-|\sigma|} \right)^{-*}_n    &    ~  ;  ~ \sigma \otimes f \mapsto  (-1)^{|\sigma| (n+|\sigma|)}  \beta(f) \\
 = (T^2C)_n(\sigma)&    ~   ;  ~  (-1)^{|\sigma| (n+|\sigma|)} (-1)^{n+|\sigma|}f(\rho \otimes a) \\
                           &    ~    ~   = (-1)^{(1+|\sigma|) (n+|\sigma|)}  f(\rho \otimes a) .\\
\end{align*}

\noindent 2. Recall, from the proof of Proposition \ref{cd} that there is an isomorphism $\tau$ 
\begin{align*} 
\Hom(TC,[TC])  &  \xleftarrow{\cong} (\Delta K \otimes [C]) \otimes [TC]\\
 &\cong (\Delta K \otimes [TC]) \otimes [C]  \\
  &  \xrightarrow{\cong} \Hom(T^2C,[C]).  
\end{align*}
and that $e_C$ is represented by the image of $\id_{TC}$ under $\tau$ (after taking colimits and 0-cycles).

However, to incorporate the triple tensor product, we factor the last isomorphism as follows:
\begin{align*} 
    (\Delta K \otimes [TC]) \otimes [C] 
  &  \cong(\Delta K^{-*} \otimes \Delta K \otimes [C])^{-*} \otimes [C] \\
      & \xrightarrow{\cong} \Hom(((\Delta K^{-*} \otimes \Delta K \otimes [C])^{-*})^{-*}, [C])\\
    & \xrightarrow{\cong} \Hom(\Delta K^{-*} \otimes \Delta K \otimes [C], [C])\\
  &  \xleftarrow{\cong} \Hom(T^2C,[C]).  
\end{align*}
where the first isomorphism is induced by 
$$( \Psi^{-1})^{-*} : (\Delta K^{-*}  \otimes \Delta K  \otimes [C])^{-*}  \to \Delta K \otimes [TC] = \Delta K \otimes  [(\Delta K \otimes [C])^{-*}] $$
 and the last isomorphism is induced by $\Psi$.
Since the last isomorphism is induced by $\Psi$, it suffices to trace through the above isomorphisms and show that the image of $\id_{TC}$ in $\Hom(\Delta K^{-*} \otimes \Delta K \otimes [C], [C])$ is represented by $\Phi$.

\medskip

First note that 
\begin{align*}
 \Hom(TC,[TC])_0(\sigma)  & = \bigoplus_{\sigma \leq \rho} \Hom(TC(\sigma), TC(\rho))_0 \\
 &\xrightarrow{\cong} \bigoplus_{\substack{ \sigma \leq \sigma ' \\
 \sigma \leq \rho \leq \tau }} \Hom(\Delta K^{-*}(\sigma) \otimes C^{-*}(\sigma '), \Delta K^{-*}(\rho) \otimes C^{-*}(\tau))_0  
 \end{align*}
The components of the image of $\id_{TC}(\sigma)$ are  identity maps when $(\sigma, \sigma ')  = (\rho, \tau)$ and are zero otherwise.  
That is,
\begin{align*}
\bigoplus_{\sigma \leq \rho} \Hom(TC(\sigma), TC(\rho))_0  & \xrightarrow{\cong} \bigoplus_{\substack{ \sigma \leq \sigma ' \\
 \sigma \leq \rho \leq \tau }}  \Hom(\Delta K^{-*}(\sigma) \otimes C^{-*}(\sigma '), \Delta K^{-*}(\rho) \otimes C^{-*}(\tau))_0 \\
  \id_{TC}(\sigma) &\mapsto \left( (\id:  \widehat{\sigma} \otimes \widehat{a} \mapsto  \widehat{\sigma} \otimes \widehat{a})_{(\sigma, \sigma ')  = (\rho, \tau)}, (0)_{(\sigma, \sigma ')  \not = (\rho, \tau)}\right) \\
\end{align*}

In the next steps, we will need the following result.  For an abelian group $A$, there is the evaluation homomorphism 
\begin{align*}
 \ev: A^* \otimes A & \to \Hom_\Z (A,A) \\
 f \otimes y & \mapsto (x \mapsto f(x)y)
\end{align*}
This is an isomorphism if $A$ is finitely generated free, and if $\{e_i\}$ is a basis for $A$ with dual basis $\{\wh e_i\}$, then $\sum_i \ev(\wh e_i \otimes e_i) = \id_A$.   

If we pass to a chain complex of abelian groups, the evaluation homomorphism is the same up to signs
\begin{align*}
 \ev: C^{-*}\otimes C  & \to \Hom (C,C) \\
 f \otimes y & \mapsto (x \mapsto (-1)^{|x||y|} f(x)y).
\end{align*}

Using these facts, and writing $a \in C(\sigma')$, and $\{ e_i \}$ as a basis for $C(\sigma')$,
 \begin{align*}
& \Hom(\Delta K^{-*}(\sigma) \otimes C^{-*}(\sigma '), \Delta K^{-*}(\sigma) \otimes C^{-*}(\sigma'))_0  ~;~ (\id:  \widehat{\sigma} \otimes \widehat{a} \mapsto  \widehat{\sigma} \otimes \widehat{a})\\
 \xleftarrow{\cong}  
 &  ( \Delta K (\sigma) \otimes C(\sigma') \otimes \Delta K^{-*}(\sigma) \otimes C^{-*}(\sigma') )_0 ~;~  \sum_i (-1)^{|\sigma| (|\sigma| + |e_i|)} \sigma \otimes e_i \otimes (\widehat{\sigma} \otimes \widehat{e_i}) \\
 \xrightarrow{\sw}  
 &    (\Delta K (\sigma)  \otimes \Delta K^{-*}(\sigma) \otimes C^{-*}(\sigma') \otimes C(\sigma'))_0 ~;~ \sum_i (-1)^{(|\sigma|-|e_i|) (|\sigma| + |e_i|)} \sigma \otimes (\widehat{\sigma} \otimes \widehat{e_i})  \otimes e_i \\
 \xrightarrow{\cong}
 &   \Hom(  \Delta K^{-*} (\sigma)  \otimes \Delta K(\sigma) \otimes C(\sigma'), C(\sigma'))_0 ~;~ ( \Phi: \widehat{\sigma} \otimes \sigma \otimes a  \mapsto a ).
 \end{align*} 
 
 Note that the signs occurring at each step in these isomorphisms come from a careful application of the signs in Lemma \ref{rules}.    
 
\end{proof}

The analogous statement of Theorem \ref{triangle} for the $\Z(K)\Mod$ category is as follows.

\begin{theorem} \label{formula-eC-K} The following assertions are valid.
\begin{enumerate}

 \item There is an isomorphism  of $\Z(K)$-chain complexes $$\Psi:  \Delta K \otimes \Delta K^{-*} \otimes [C] \to T^2C =  (\Delta K^{-*} \otimes  [(\Delta K^{-*} \otimes [C])^{-*}] )^{-*} $$
given by 
\begin{multline*}
\Psi_n(\tau):  \bigoplus_{\sigma \leq \rho \leq \tau}  ( \Delta K(\tau) \otimes \Delta K^{-*}(\rho) \otimes C(\sigma))_n \to ((\Delta K^{-*} \otimes  [(\Delta K^{-*} \otimes [C])^{-*}] )^{-*} )_n(\tau)
\end{multline*}
where
 $$
 \Psi_n(\tau)(\tau \otimes \wh  \rho \otimes a) =  \left(\tau \otimes f \mapsto (-1)^{(1+|\tau|) (n+|\tau|)}  f(\wh \rho \otimes a)\right).
 $$

 \item
If $\Phi_n(\tau' \leq \tau)$ is the map defined by the commutative triangle:
$$
\begin{tikzcd}
  \bigoplus_{\sigma \leq \rho \leq \tau} (  \Delta K(\tau) \otimes \Delta K^{-*}(\rho) \otimes C(\sigma))_n  \arrow[rd,"\Phi_n(\tau' \leq \tau)"] \arrow[dd,"\Psi_n(\tau)" ']   &\\
 & C_n(\tau') \\
 ( T^2C)_n(\tau)\arrow[ru,"(e_C)_n(\tau' \leq \tau) "'] &
\end{tikzcd}
$$

\medskip

then

\begin{itemize}
 \item $\Phi_n(\tau' \leq \tau)|_{  \Delta K(\tau) \otimes \Delta K^{-*}(\rho) \otimes C(\sigma)} = 0$ if $\tau <\rho$.
 \item $\Phi_n(\tau' \leq \tau)|_{ \Delta K(\tau) \otimes \Delta K^{-*}(\rho) \otimes C(\sigma)} = 0$ if $\sigma \neq \tau'$.
 \item  $\Phi_n(\tau' \leq \tau)(\tau  \otimes   \wh \tau  \otimes   a) =  (-1)^{| \tau|} a$, with $\tau$ a generator of the infinite cyclic group $\Delta K_{|\tau|}(\tau)$ and  $\wh \tau$ the dual generator of $\Delta K^{|\tau|}(\tau )$   and $a \in   C_n(\tau')$.
   \end{itemize}
  
  \end{enumerate}

\end{theorem}

\begin{proof}  
The proof is similar to that of Theorem \ref{triangle}.
\end{proof}

A $\Z[DK]$-module $M(\sigma \leq \tau)$ extends to a $\Z[K^{\op} \times K]$-module as follows:
$$
M(\sigma,\tau) = 
\begin{cases}
M(\sigma\leq\tau) & \text{ if } \sigma \leq \tau\\
0  & \text{ if } \sigma \not \leq \tau\\
\end{cases}
$$

\begin{corollary}  \label{cdvi}
If $C \in \ch(\Z(K^{\op})\Mod)$ and $\sigma \in K$, there is a commutative square with vertical isomorphisms
$$
\begin{tikzcd}
 {[}\Delta K^{-*} \otimes_{\Z(K)} \Delta K{]}(\sigma,-)  \otimes_{\Z[K^{\op}]} [C] \arrow[r,"\varepsilon \otimes \id"] \arrow[d,"{[}\Psi{]}(\sigma)"] & \ul \Z^{DK}(\sigma, -) \otimes_{\Z[K^{\op}]} [C] \arrow[d,"\YL"] \\
 {[}T^2C{]}(\sigma) \arrow[r,"{[}e_C{]}(\sigma)"]  & {[}C{]}(\sigma)
\end{tikzcd}
$$
\end{corollary}

\begin{proof}
  The right hand vertical map is given by Yoneda's Lemma.    The left hand vertical isomorphism is defined to be $[\Psi](\sigma)
  = \bigoplus_{\sigma \leq x} \Psi(x)$ from Theorem \ref{triangle}, after we identify

\begin{align*}
 {[}\Delta K^{-*} \otimes_{\Z(K)} \Delta K{]}
(\sigma, -) 
 \otimes_{\Z[K^{\op}]} [C] & = \bigoplus_{\sigma \leq \tau} {[}\Delta K^{-*} \otimes_{\Z(K)} \Delta K{]}(\sigma \leq \tau)  \otimes_{\Z} C(\tau)    \\
 & =  \bigoplus_{\sigma \leq x \leq y \leq \tau} (\Delta K^{-*}(x) \otimes_{\Z} \Delta K(y)  \otimes_{\Z} C(\tau))\\
 & =  [\Delta K^{-*} \otimes \Delta K  \otimes C] (\sigma)
\end{align*} 
We  view the square as a triangle with a vertex at the upper left  and two vertices on the bottom.   Then the triangle (and hence the square commutes) since it results from applying $[ ~ ]$ to the commutative triangle from Theorem \ref{triangle}.

\end{proof}

The analogous statement of Corollary \ref{cdvi} for the $\Z(K)\Mod$ category is as follows.

\begin{corollary}  \label{cdvi op}
Let $C \in \ch(\Z(K)\Mod)$ and $\tau \in K$, there is a commutative square with vertical isomorphisms
$$
\begin{tikzcd}
 {[}\Delta K \otimes_{\Z(K^{\op})} \Delta K^{-*}{]}(\tau, -
)  \otimes_{\Z[K]} [C] \arrow[r,"\varepsilon \otimes \id"] \arrow[d,"{[}\Psi{]}(\tau)"] & \ul \Z^{DK^{\op}}(\tau,-) \otimes_{\Z[K]} [C] \arrow[d,"\YL"] \\
 {[}T^2C{]}(\tau) \arrow[r,"{[}e_C{]}(
\tau)"]  & {[}C{]}(\tau)
\end{tikzcd}
$$
\end{corollary}
\begin{proof} The proof is similar to the one for Corollary \ref{cdvi}.
\end{proof}

\begin{corollary}
The chain map $e_C$ is a chain homotopy equivalence.
\end{corollary}

\begin{proof}
Corollaries \ref{varepsilon is a weak eq} and \ref{varepsilon is a weak eq op} imply that, for all $\begin{cases} \sigma \\ \tau\end{cases}$,

$$
\begin{cases} [\Delta K^{-*} \otimes_{\Z(K)} \Delta K{]}
(\sigma, -)  \to \ul \Z^{K} \\  [\Delta K \otimes_{\Z(K^{\op})} \Delta K^{-*}{]}
(\tau, -)  \to \ul \Z^{K^{\op}}
\end{cases}
$$ 
is a weak equivalence of $\begin{cases} \Z[K] \\  \Z[K^{\op}]\end{cases}$-modules.  Thus by  Corollaries
\ref{cdvi} and \ref{cdvi op} and Proposition \ref{whe}, $[e_C](\sigma)$ and  $[e_C](\tau)$ are a weak equivalences.  Then by Proposition \ref{whe=che}, $e_C$ is a chain homotopy equivalence.
\end{proof}

This completes the proof of our  Theorem \ref{main}.

\subsection*{Acknowledgements}Both authors wish to thank the Max Planck Institute in Bonn for its hospitality during the time in which much of this research was carried out.   They also wish to thank Frank Connolly and Tibor Macko for helpful correspondence and the two anonymous referees for their extensive comments on the paper. 

\subsection*{Funding}
The first named author wishes to acknowledge support from the National Science Foundation DMS 1615056 and the Simons Foundation Collaboration Grant 713226.

The second named author wishes to acknowledge the support by the Deutsche Forschungsgemeinschaft (DFG, German Research Foundation) under Germany's Excellence Strategy EXC2181/1-390900948 (the Heidelberg STRUCTURES Excellence Cluster), the Deutsche Forschungsgemeinschaft 281869850 (RTG 2229).

\bibliography{refs}

\begin{thebibliography}{KMM13}

\bibitem[AM18]{Spiros-Tibor}
Spiros Adams{-}Florou and Tibor Macko.
\newblock {$L$}-homology on ball complexes and products.
\newblock {\em Homology Homotopy Appl.}, 20(2):11--40, 2018.

\bibitem[Con22]{Frank}
Frank Connolly.
\newblock A geometric interpretation of {R}anicki duality,
  \href{https://arxiv.org/abs/2203.10160}{https://arxiv.org/abs/2203.10160},
  2022.

\bibitem[DL98]{DL98}
James~F. Davis and Wolfgang L{\"u}ck.
\newblock Spaces over a category and assembly maps in isomorphism conjectures
  in {$K$}- and {$L$}-theory.
\newblock {\em $K$-Theory}, 15(3):201--252, 1998.

\bibitem[Dol80]{Dold}
Albrecht Dold.
\newblock {\em Lectures on algebraic topology}, volume 200 of {\em Grundlehren
  der Mathematischen Wissenschaften [Fundamental Principles of Mathematical
  Sciences]}.
\newblock Springer-Verlag, Berlin-New York, second edition, 1980.

\bibitem[KM63]{KM}
Michel~A. Kervaire and John~W. Milnor.
\newblock Groups of homotopy spheres. {I}.
\newblock {\em Ann. of Math. (2)}, 77:504--537, 1963.

\bibitem[KMM13]{KMM}
Philipp K\"{u}hl, Tibor Macko, and Adam Mole.
\newblock The total surgery obstruction revisited.
\newblock {\em M\"{u}nster J. Math.}, 6(1):181--269, 2013.

\bibitem[KS77]{KS}
Robion~C. Kirby and Laurence~C. Siebenmann.
\newblock {\em Foundational essays on topological manifolds, smoothings, and
  triangulations}.
\newblock Annals of Mathematics Studies, No. 88. Princeton University Press,
  Princeton, N.J.; University of Tokyo Press, Tokyo, 1977.
\newblock With notes by John Milnor and Michael Atiyah.

\bibitem[Mas91]{Ma91}
William~S. Massey.
\newblock {\em A basic course in algebraic topology}, volume 127 of {\em
  Graduate Texts in Mathematics}.
\newblock Springer-Verlag, New York, 1991.

\bibitem[Pal15]{P15}
Christopher Palmer.
\newblock {\em Some applications of algebraic surgery theory: 4-manifolds,
  triangular matrix rings and braids}.
\newblock PhD thesis, University of Edinburgh, 2015.

\bibitem[Qui70]{Quinn}
Frank~Stringfellow Quinn, III.
\newblock {\em A {GEOMETRIC} {FORMULATION} {OF} {SURGERY}}.
\newblock ProQuest LLC, Ann Arbor, MI, 1970.
\newblock Thesis (Ph.D.)--Princeton University.

\bibitem[Ran80]{ats1}
A.~A. Ranicki.
\newblock The algebraic theory of surgery. {I}. {F}oundations.
\newblock {\em Proc. London Math. Soc. (3)}, 40(1):87--192, 1980.

\bibitem[Ran92]{bluebook}
A.~A. Ranicki.
\newblock {\em Algebraic {$L$}-theory and topological manifolds}, volume 102 of
  {\em Cambridge Tracts in Mathematics}.
\newblock Cambridge University Press, Cambridge, 1992.

\bibitem[RW90]{RW}
Andrew Ranicki and Michael Weiss.
\newblock Chain complexes and assembly.
\newblock {\em Math. Z.}, 204(2):157--185, 1990.

\bibitem[Spa81]{S66}
Edwin~H. Spanier.
\newblock {\em Algebraic topology}.
\newblock Springer-Verlag, New York-Berlin, 1981.
\newblock Corrected reprint.

\bibitem[Sul66]{Sullivan-thesis}
Dennis Sullivan.
\newblock Triangulating homotopy equivalences.
\newblock {\em Princeton Ph.D. Thesis}, 1966.

\bibitem[Sul96]{Sullivan-Notes}
D.~P. Sullivan.
\newblock Triangulating and smoothing homotopy equivalences and homeomorphisms.
  {G}eometric {T}opology {S}eminar {N}otes (1967).
\newblock In {\em The {H}auptvermutung book}, volume~1 of {\em $K$-Monogr.
  Math.}, pages 69--103. Kluwer Acad. Publ., Dordrecht, 1996.

\bibitem[Wal99]{Wall}
C.~T.~C. Wall.
\newblock {\em Surgery on compact manifolds}, volume~69 of {\em Mathematical
  Surveys and Monographs}.
\newblock American Mathematical Society, Providence, RI, second edition, 1999.
\newblock Edited and with a foreword by A. A. Ranicki.

\bibitem[WW95]{WW}
Michael Weiss and Bruce Williams.
\newblock Assembly.
\newblock In {\em Novikov conjectures, index theorems and rigidity, {V}ol. 2
  ({O}berwolfach, 1993)}, volume 227 of {\em London Math. Soc. Lecture Note
  Ser.}, pages 332--352. Cambridge Univ. Press, Cambridge, 1995.

\end{thebibliography}
\bibliographystyle{alpha}

\Addresses

\end{document}